\documentclass[11pt,twoside]{amsart}

\usepackage{amssymb}
\usepackage{amsthm}
\usepackage{amsmath}
\usepackage{graphics}
\usepackage{tikz}
\usetikzlibrary{arrows}
\usetikzlibrary{shapes,snakes}

\usepackage{xparse,aliascnt,hyperref,bookmark}

\NewDocumentCommand{\xnewtheorem}{m o m}
 {%
  \IfNoValueTF{#2}
   {\newtheorem{#1}{#3}}
   {%
    \newaliascnt{#1}{#2}%
    \newtheorem{#1}[#1]{#3}%
    \aliascntresetthe{#1}%
    \expandafter\newcommand\csname #1autorefname\endcsname{#3}%
   }%
 }

\tikzstyle{block} = [draw,fill=blue!20,minimum size=0.5em]

\tikzstyle{branch}=[fill,shape=circle,minimum size=3pt,inner sep=0pt]

\xnewtheorem{thm}{base}
\xnewtheorem{tm}[thm]{Theorem}
\xnewtheorem{open}[thm]{Open Problem}
\xnewtheorem{conj}[thm]{Conjecture}
\xnewtheorem{lem}[thm]{Lemma}
\xnewtheorem{prop}[thm]{Proposition}
\xnewtheorem{cor}[thm]{Corollary}
\xnewtheorem{claim}[thm]{Claim}
\xnewtheorem{remark}[thm]{Remark}

\newcommand{\ignore}[1]{}
\newcommand{\set}[1]{\left\{#1\right\}}
\newcommand{\conv}[1]{\mbox{\rm{conv}}#1}
\newcommand{\supp}[1]{\mbox{\rm{supp}}#1}

\def\ba{\begin{array}}
\def\ea{\end{array}}
\def\beq{\begin{equation}}
\def\eeq{\end{equation}}
\def\bea{\begin{eqnarray}}
\def\eea{\end{eqnarray}}
\def\beann{\begin{eqnarray*}}
\def\eeann{\end{eqnarray*}}
\def\tn{\textnormal}
\def\A{\mathcal{A}}
\def\F{\mathcal{F}}
\def\M{\mathcal{M}}

\def\O{\mathcal{O}}

\def\S{\mathcal{S}}
\def\T{\mathcal{T}}

\def\W{\mathcal{W}}
\def\tn{\textnormal}

\def\es{\emptyset}

\def\sm{\setminus}

\def\het{\hat}
\def\a{\alpha}
\def\b{\beta}
\def\l{\ell}

\DeclareMathOperator{\SA}{SA}
\DeclareMathOperator{\LS}{LS}
\DeclareMathOperator{\BZ}{BZ}
\DeclareMathOperator{\Las}{Las}
\DeclareMathOperator{\BCC}{BCC}

\newlength{\mylen}

\title[A Comprehensive Analysis of Polyhedral Lift-and-Project Methods]{\bf A Comprehensive Analysis\\
of\\
Polyhedral Lift-and-Project Methods}
\thanks{Some of the material in this manuscript appeared in a preliminary
form in IPCO 2011 Proceedings, see \cite{AT2011} and
in the first author's PhD Thesis \cite{Au2014}.}

\author{Yu Hin Au}
\thanks{Yu Hin Au: Research of this author was supported in part by a Tutte
Scholarship, a Sinclair Scholarship, an NSERC scholarship, research grants
from University of Waterloo and Discovery Grants from NSERC. Department of Mathematics, Milwaukee School of Engineering, Milwaukee, Wisconsin, U.S.A. E-mail: au@msoe.edu}

\author{Levent Tun\c{c}el}
\thanks{Levent Tun\c{c}el: Research of this author was supported in part by research grants
from University of Waterloo and Discovery Grants from NSERC. Department of Combinatorics and
Optimization, Faculty of Mathematics, University of Waterloo, Waterloo, Ontario, N2L 3G1 Canada. E-mail: ltuncel@uwaterloo.ca}

\date{\today}

\keywords{combinatorial optimization, lift-and-project methods, design and analysis of algorithms
with discrete structures, integer programming, semidefinite programming, convex relaxations}

\begin{document}
\maketitle              

\begin{abstract}
We consider lift-and-project methods for combinatorial optimization problems and
focus mostly on those lift-and-project methods which generate polyhedral relaxations
of the convex hull of integer solutions.  We introduce many new variants
of Sherali--Adams and Bienstock--Zuckerberg operators.  These new operators
fill the spectrum of polyhedral lift-and-project operators in a way which
makes all of them more transparent, easier to relate to each other,
and easier to analyze.  We provide new techniques to analyze the
worst-case performances as well as relative strengths of these operators
in a unified way.  In particular, using the new techniques and a result of Mathieu and Sinclair from 2009, we prove that the polyhedral Bienstock--Zuckerberg operator requires at least
$\sqrt{2n}- \frac{3}{2}$ iterations to compute the matching polytope of the $(2n+1)$-clique.
We further prove that the operator requires approximately $\frac{n}{2}$ iterations to reach the stable set polytope of the $n$-clique,
if we start with the fractional stable set polytope. Lastly, we
show that some of the worst-case instances for the positive semidefinite
Lov\'{a}sz--Schrijver lift-and-project operator are also bad instances
for the strongest variants of the Sherali--Adams operator with positive semidefinite strengthenings, and discuss some consequences for integrality gaps of convex relaxations.
\end{abstract}

\section{Introduction}

Given a polytope $P \subseteq [0,1]^n$, we are interested in its integer hull (i.e., the convex hull of $0,1$ vectors in $P$), $P_I := \tn{conv}\left(P \cap \set{0,1}^n\right)$. While it is impossible to efficiently find
a description of $P_I$ for a general $P$ (unless $\mathcal{P} = \mathcal{NP}$),
we may use properties that we know are satisfied by points in $P_I$ to derive inequalities that are valid for $P_I$ but not $P$.

\emph{Lift-and-Project} methods provide a systematic way to generate a sequence of convex relaxations of $P_I$,
converging to the integer hull $P_I$.  These methods
go back to work by Balas and others in the late 1960s and the early 1970s. Some of the most attractive features of
these methods are:
\begin{itemize}
\item Convex relaxations of $P_I$ obtained after $O(1)$ iterations of the procedure are
\emph{tractable} provided $P$ is tractable.  Here, tractable may mean either that the
underlying linear optimization problem is polynomial-time solvable, say due to the existence
of a polynomial-time weak separation oracle for $P$; or, more strongly, that $P$ has an
explicitly given, polynomial size representation by linear inequalities (we will distinguish between
these two versions of tractability, starting with the strength chart given in~\autoref{fig0}).
\item
Many of these methods use \emph{lifted} (higher dimensional) representations for the relaxations.
Such representations sometimes allow compact (polynomial size in the input) convex representations
of exponentially many facets.
\item
Most of these methods allow easy addition of positive semidefiniteness constraints in the lifted space.
This feature can make the relaxations much stronger in some cases, without sacrificing
polynomial-time solvability (perhaps only approximately).
Moreover, these semidefiniteness constraints can represent an uncountable family of defining linear inequalities,
such as those of the theta body of a graph.
\item
Systematic generation of tighter and tighter relaxations converging to $P_I$ in at most $n$ rounds makes
the strongest of these methods good candidates for utilization in generating polynomial-time
approximation algorithms for hard problems, or for proving large integrality gaps
(hence providing a negative result about approximability in the underlying hierarchy of relaxations).
\end{itemize}

In the last two decades, many lift-and-project operators have been proposed (see, for example,
\cite{SheraliA90a}, \cite{LovaszS91a}, \cite{BalasCC93a}, \cite{Lasserre01a} and~\cite{BienstockZ04a}), and have been applied to various discrete optimization problems
(see, for example, \cite{SL97}, \cite{deKP2002}, \cite{PVZ07} and~\cite{GL07}).  Many families of facets of the stable set polytope of graphs are shown to be easily generated by these procedures~\cite{LovaszS91a, LiptakT03a}. Also studied are their performances on  max-cut~\cite{Laurent02a}, set covering~\cite{BienstockZ04a}, $k$-constraint satisfiability problems~\cite{Schoenebeck08a}, knapsack~\cite{KarlinMN11a}, sparsest cut~\cite{GuptaTW13a}, directed Steiner tree~\cite{FriggstadKKLS14a}, set partitioning~\cite{SL97},  TSP relaxations~\cite{CookD2001a, Cheung05a, CheriyanGGS13a}, and matching~\cite{StephenT99a, AguileraBN04a, MathieuS09a}.  For general properties of these operators and some comparisons among them, see~\cite{GoemansT01a}, \cite{Laurent03a}  and \cite{HongT08a}.

\begin{figure}[htb]
\begin{center}
\begin{tikzpicture}[y=1.1cm, x=1.2cm,>=latex', main node/.style={}, word node/.style={font=\footnotesize}]
\def\x{0}

\node at (0,0) (BCC) {$\BCC$};
\node at (1.5,0) (LS_0) {$\LS_0$};
\node at (3,0) (LS) {$\LS$};
\node at (4.5,0) (SA) {$\SA$};
\node[main node] at (6,0) (SA') {$\mathbf{SA'}$};
\node at (7.5,0) (BZ) {$\BZ$};
\node[main node] at (9,0) (BZ'') {$\mathbf{BZ''}$};
\node[main node] at (10.5,0) (BZ') {$\mathbf{BZ'}$};
\node at (3,1) (LS_+) {$\LS_+$};
\node[main node] at (4.5,1) (SA_+) {$\mathbf{SA_+}$};
\node[main node] at (6,1) (SA'_+) {$\mathbf{SA'_+}$};
\node at (7.5,1) (BZ_+) {$\BZ_+$};
\node[main node] at (9,1) (BZ''_+) {$\mathbf{BZ''_+}$};
\node[main node] at (10.5,1) (BZ'_+) {$\mathbf{BZ_+'}$};
\node at (7.5,2) (Las) {$\Las$};

        \draw[->] (BCC) -- (LS_0);
        \draw[->] (LS_0) -- (LS);
        \draw[->] (LS) -- (SA);
        \draw[->] (SA) -- (SA');
        \draw[->] (SA') -- (BZ);
        \draw[->] (BZ) -- (BZ'');
        \draw[->] (BZ'') -- (BZ');

        \draw[->] (LS_+) -- (SA_+);
        \draw[->] (SA_+) -- (SA'_+);
        \draw[->] (SA'_+) -- (BZ_+);
        \draw[->] (BZ_+) -- (BZ''_+);
        \draw[->] (BZ''_+) -- (BZ'_+);

        \draw[->] (LS) -- (LS_+);
        \draw[->] (SA) -- (SA_+);
        \draw[->] (SA') -- (SA'_+);
        \draw[->] (BZ) -- (BZ_+);
        \draw[->] (BZ') -- (BZ'_+);
        \draw[->] (BZ'') -- (BZ''_+);

        \draw[->] (SA_+) to [out= 45, in = 180] (Las);

\def\y{-1} 

\draw[dotted,thick] (-0.5,0.5) -- (12.7,0.5) ;
\draw[dotted,thick] (6.7,\y - 0.3) -- (6.7,2.5) ;
\draw[dotted,thick] (9.7,\y - 1.3) -- (9.7,2.5) ;

\node[word node, align=left, above] at (12,0.5)%
{PSD\\ Operators};
\node[word node, align=left, below] at (12,0.5)%
{Polyhedral\\ Operators};

\node[word node, align=left] at  (3.35,\y) (tract1)%
{Tractable w/ weak separation oracle for $P$};

\node[word node, align=left] at  (4.85, \y - 1) (tract2)%
{Tractable w/ facet description of $P$};

\node at (-0.5,\y) (n1) {};
\node at (6.7,\y) (n2) {};
\node at (-0.5,\y-1) (n3) {};
\node at (9.7,\y-1) (n4) {};

\path[->|]
(tract1) edge (n1)
(tract1) edge (n2)
(tract2) edge (n3)
(tract2) edge (n4);

\end{tikzpicture}
\end{center}
\caption{A strength chart of lift-and-project operators.}\label{fig0}
\end{figure}
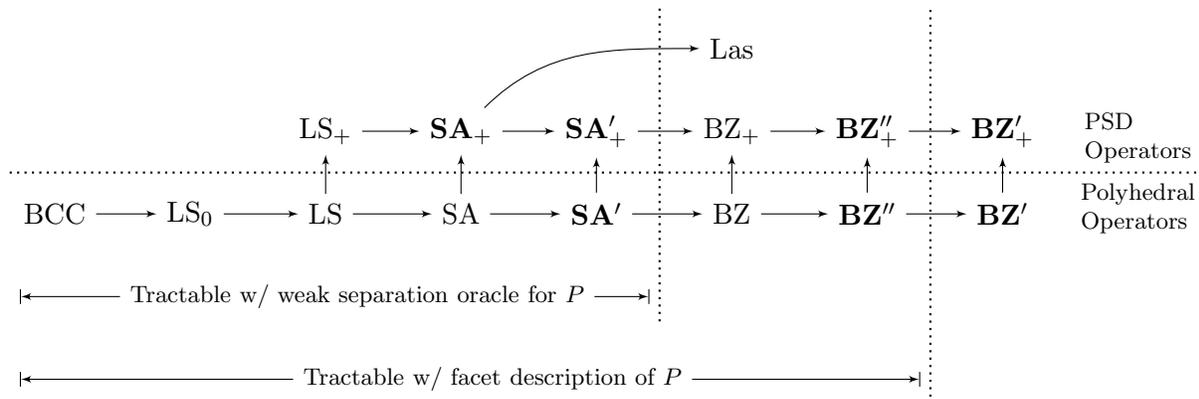

\autoref{fig0} provides a glimpse of the spectrum of polyhedral lift-and-project operators, as well as their
semidefinite strengthened counterparts. The operators $\BCC$ (due to Balas, Ceria and Cornu{\'e}jols~\cite{BalasCC93a}); $\LS_0, \LS$ and $\LS_+$ (due to Lov\'{a}sz and Schrijver~\cite{LovaszS91a}); $\SA$ (due to Sherali and Adams~\cite{SheraliA90a}); and $\BZ, \BZ_+$ (due to Bienstock and Zuckerberg~\cite{BienstockZ04a}) will be formally defined in the subsequent sections. The $\Las$ operator is due to Lasserre~\cite{Lasserre01a}. The boldfaced operators in the figure are the new ones proposed in the current paper, and each solid arrow in the chart denotes ``is dominated by'' (i.e., the operator that is at the head of an arrow is stronger than that at the tail). For instance, when applied to the same set $P$, the $\LS_0$ operator yields a relaxation that is at least as tight as that obtained by applying the $\BCC$ operator.

Observe that $\BCC$ is dominated by every other operator in~\autoref{fig0}. Since $\BCC$ admits a very short and elegant proof that it returns $P_I$ after $n$ iterations for every $P \subseteq [0,1]^n$, it follows immediately that \emph{every} operator in \autoref{fig0} converges to $P_I$ in at most $n$ iterations. More generally, if one can prove an upper-bound result for any operator $\Gamma$ in \autoref{fig0}, then the same result applies to all operators in the diagram that can be reached from $\Gamma$ by a directed path. On the other hand, any lower-bound result on the $\BZ'$ operator implies the same result for all polyhedral lift-and-project operators in \autoref{fig0}. Likewise, to obtain a lower bound result for all lift-and-project operators shown in the diagram, it suffices to show that the result holds for $\BZ'_+$ and $\Las$. (For some bad instances for $\Las$, see~\cite{Laurent02a} and~\cite{Cheung07a}. See also~\cite{Schoenebeck08a} and~\cite{Tulsiani09a} for some integrality gap results on $\Las$ relaxations.)

As seen in \autoref{fig0}, the strongest polyhedral lift-and-project operators known to date are
$\LS, \SA$ and $\BZ$. We are interested in these strongest operators because they provide the
 strongest tractable relaxations obtained this way. On the other hand, if we want to prove that some combinatorial optimization problem is difficult to attack by lift-and-project methods, then we would hope to establish them on the strongest existing hierarchy for the strongest negative results. For example, some of the non-approximability results on vertex cover are based on the $\LS_+$ operator~\cite{GeorgiouMPT10a, SchoenebeckTT07a}, and some other integrality gap results are based on $\SA$~\cite{CMM09}.

Furthermore, it was shown in~\cite{ChanLRS13a} that non-approximability results for the $\SA$ relaxations of approximate constraint satisfaction problems can be extended to lower bound results on the extension complexity (i.e., the smallest number of variables needed to represent a given set as the projection of a tractable set in higher dimension) of the max-cut and max 3-sat polytopes. The reader may refer to~\cite{Yannakakis91a} for the first major progress on the extension complexity of polytopes that arise from combinatorial optimization problems, and~\cite{Goemans09a, FioriniMPTW12a, Rothvoss14a} for some of the recent breakthroughs in this line of work.

Therefore, by understanding the more powerful lift-and-project operators, we could either obtain better approximations for hard combinatorial optimization problems, or lay some of the groundwork for yet stronger non-approximability results. Moreover, we shall see that these analyses typically also lead to other crucial information about the underlying hierarchy of convex relaxations, such as their integrality gaps.

This paper will be organized as follows. In Section 2, we introduce many of the existing lift-and-project methods, as well as $\SA'$ and $\BZ'$ --- strengthened variants of $\SA$ and $\BZ$, respectively. In particular, $\BZ$ is a substantial procedure with many complicated details, and we believe that our version $\BZ'$ is simpler to present and analyze. We will mostly use $\BZ'$ to establish lower-bound results. Since $\BZ'$ dominates $\BZ$, it follows that all of these lower-bound results also apply to $\BZ$.  We shall also see that these operators can all be seen as lifting to sets of matrices whose rows and columns indexed by subsets of $\set{0,1}^n$, a framework exposed by Lov\'{a}sz and Schrijver~\cite{LovaszS91a} and extensively used by Bienstock and Zuckerberg~\cite{BienstockZ04a}.

In Section 3, we introduce notions such as admissible lift-and-project operators and measure consistency for matrices and vectors, and identify situations in which some variables in the lifted space do not help generate cuts. This provides a template that can streamline the analyses of the worst-case performances as well as relative strengths of various lift-and-project methods. We show that, under certain conditions, the performance of $\SA'$ and $\BZ$ are closely related to each other. Since $\SA'$ inherits many properties from the well-studied $\SA$ operator, this connection provides another venue to understanding and analyzing $\BZ$. Next, we utilize the tools we have established and prove that the $\BZ$ operator requires at least $\sqrt{2n}-\frac{3}{2}$ iterations to compute the matching polytope of the $(2n+1)$-clique, and approximately $\frac{n}{2}$ iterations to compute the stable set polytope of the $n$-clique. This establishes the first examples in which $\BZ$ requires more than $O(1)$ iterations to reach the integer hull.

Next, in Section 4, we turn our focus to lift-and-project operators that utilize positive semidefiniteness constraints. We construct two strong, semidefinite versions of the Sherali--Adams operator that we call $\SA_+$ and $\SA_+'$. There are other weaker versions of these operators in the recent literature called \emph{Sherali--Adams SDP} which have been previously studied, among others, by Chlamtac and Singh~\cite{ChlamtacS08a} and Benabbas et al.~\cite{BenabbasGM10a, BenabbasM10a, BenabbasCGM11a, BenabbasGM12a}, even though our versions are the strongest yet. Using techniques developed in Section 3, we relate the performance of $\SA_+'$ and $\BZ_+$ (the $\BZ$ operator enhanced with an additional positive semidefiniteness constraint) under certain conditions. Next, we develop some tools for proving upper-bound results, and show that $\SA_+'$ and $\BZ_+'$ (a strengthened and simplified version of $\BZ_+$) require at most $n - \left\lfloor \frac{ \sqrt{2n+1} -1}{2} \right\rfloor$ and $\left\lceil \sqrt{2n+ \frac{1}{4}} - \frac{1}{2} \right\rceil$ iterations, respectively, to compute the matching polytope of the $(2n+1)$-clique. We then show that positive semdefiniteness constraints do not help in some cases, and prove that some well-known worst-case instances for $\LS$ and $\LS_+$ extend to give worst-case instances for $\SA$ and $\SA_+$.

Finally, we conclude the paper by illustrating how the analyses and the tools we provided may be used to prove integrality gaps for various classes of relaxations obtained from lift-and-project operators with some desirable invariance properties. The details of the original $\BZ$ and $\BZ_+$ operators, as well as their relationships with our new variants, are given in the Appendix.

Several of our results can be seen as ``approximate converses'' of the dominance relationship among various lift-and-project operators. Such relationships are represented by dashed arrows in~\autoref{fig2}. As we shall see, sometimes a weaker operator can be guaranteed to perform at least as well as a stronger one, by an appropriate increase of iterate number and/or certain assumptions on the given polytope $P$. These reverse dominance results, together with the new operators we define and other tools we provide, fill the spectrum of lift-and-project operators in a way which makes all of them more transparent, easier to relate to each other, and easier to analyze.

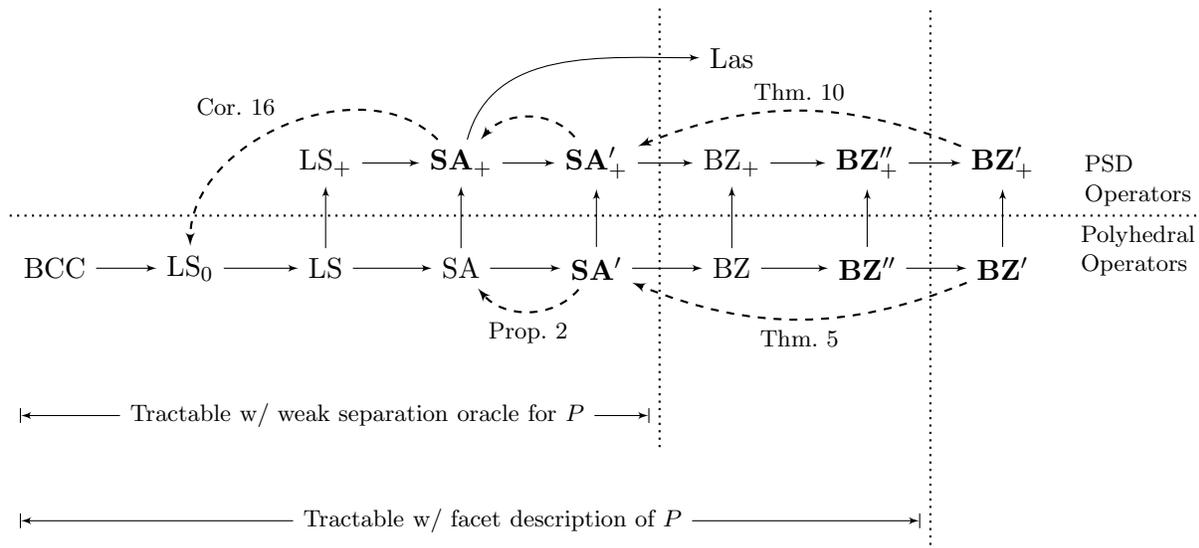
\begin{figure}[htb]
\begin{center}
\begin{tikzpicture}[y=1.4cm, x=1.2cm,>=latex', main node/.style={}, word node/.style={font=\footnotesize}]
\def\x{0}

\node at (0,0) (BCC) {$\BCC$};
\node at (1.5,0) (LS_0) {$\LS_0$};
\node at (3,0) (LS) {$\LS$};
\node at (4.5,0) (SA) {$\SA$};
\node[main node] at (6,0) (SA') {$\mathbf{SA'}$};
\node at (7.5,0) (BZ) {$\BZ$};
\node[main node] at (9,0) (BZ'') {$\mathbf{BZ''}$};
\node[main node] at (10.5,0) (BZ') {$\mathbf{BZ'}$};
\node at (3,1) (LS_+) {$\LS_+$};
\node[main node] at (4.5,1) (SA_+) {$\mathbf{SA_+}$};
\node[main node] at (6,1) (SA'_+) {$\mathbf{SA'_+}$};
\node at (7.5,1) (BZ_+) {$\BZ_+$};
\node[main node] at (9,1) (BZ''_+) {$\mathbf{BZ''_+}$};
\node[main node] at (10.5,1) (BZ'_+) {$\mathbf{BZ_+'}$};
\node at (7.5,2) (Las) {$\Las$};

        \draw[->] (BCC) -- (LS_0);
        \draw[->] (LS_0) -- (LS);
        \draw[->] (LS) -- (SA);
        \draw[->] (SA) -- (SA');
        \draw[->] (SA') -- (BZ);
        \draw[->] (BZ) -- (BZ'');
        \draw[->] (BZ'') -- (BZ');

        \draw[->] (LS_+) -- (SA_+);
        \draw[->] (SA_+) -- (SA'_+);
        \draw[->] (SA'_+) -- (BZ_+);
        \draw[->] (BZ_+) -- (BZ''_+);
        \draw[->] (BZ''_+) -- (BZ'_+);

        \draw[->] (LS) -- (LS_+);
        \draw[->] (SA) -- (SA_+);
        \draw[->] (SA') -- (SA'_+);
        \draw[->] (BZ) -- (BZ_+);
        \draw[->] (BZ') -- (BZ'_+);
        \draw[->] (BZ'') -- (BZ''_+);

        \draw[->] (SA_+) to [out= 75, in = 180] (Las);

\def\z{22} 
        \draw[->,dashed, thick] (BZ'_+) to [out=180- \z, in = \z] node [above ] {{\footnotesize Thm.\ 10}}(SA'_+);
        \draw[->,dashed, thick] (BZ') to [out=180+ \z , in= 360- \z] node [below] {{\footnotesize Thm.\ 5}}(SA');

\def\w{50}

        \draw[->,dashed, thick] (SA'_+) to [out=180- \w, in = \w]  (SA_+);
        \draw[->,dashed, thick] (SA') to [out=180+ \w , in= 360- \w] node [below] {{\footnotesize Prop.\ 2}} (SA);

        \draw[->,dashed, thick] (SA_+) to [out=135, in= 90] node [above left] {{\footnotesize Cor.\ 16}}(LS_0);

\def\y{-1.4} 

\draw[dotted,thick] (-0.5,0.5) -- (12.7,0.5) ;
\draw[dotted,thick] (6.7,\y - 0.3) -- (6.7,2.5) ;
\draw[dotted,thick] (9.7,\y - 1.3) -- (9.7,2.5) ;

\node[word node, align=left, above] at (12,0.5)%
{PSD\\ Operators};
\node[word node, align=left, below] at (12,0.5)%
{Polyhedral\\ Operators};

\node[word node, align=left] at  (3.35,\y) (tract1)%
{Tractable w/ weak separation oracle for $P$};

\node[word node, align=left] at  (4.85, \y - 1) (tract2)%
{Tractable w/ facet description of $P$};

\node at (-0.5,\y) (n1) {};
\node at (6.7,\y) (n2) {};
\node at (-0.5,\y-1) (n3) {};
\node at (9.7,\y-1) (n4) {};

\path[->|]
(tract1) edge (n1)
(tract1) edge (n2)
(tract2) edge (n3)
(tract2) edge (n4);

\end{tikzpicture}
\end{center}
\caption{An illustration of several restricted reverse dominance results (dashed arrows) in this paper.}\label{fig2}
\end{figure}

\section{Preliminaries}

In this section, we describe several lift-and-project operators that produce polyhedral relaxations, and establish some notation. 
One of the most fundamental ideas behind the lift-and-project approach is convexification, which can be traced back to Balas' work on disjunctive cuts in the 1970s. For convenience, we denote the set $\set{1, 2, \ldots,n}$ by $[n]$ herein. Observe that, given $P \subseteq [0,1]^n$, if we have mutually disjoint sets $Q_1, \ldots, Q_{\l} \subseteq P$ such that their union, $\bigcup_{i=1}^{\l} Q_i$, contains all integral points in $P$, then we can deduce that $P_I$ is contained in $\tn{conv}\left(\bigcup_{i=1}^{\l} Q_i \right)$, which therefore is a potentially tighter relaxation of $P_I$ than $P$.  Perhaps the simplest way to illustrate this idea is via the operator devised by Balas, Ceria and Cornu\'{e}jols~\cite{BalasCC93a} which we call the $\BCC$ operator. Given $P \subseteq [0,1]^n$ and an index $i \in [n]$, define
\[
\BCC_i(P) := \tn{conv}\left( \set{x \in P: x_i \in \set{0,1}}\right).
\]
Moreover, we can apply $\BCC_i$ followed by $\BCC_j$ to a polytope $P$ to make progress. In fact, it is well-known that for every $P \subseteq [0,1]^n$,
\[
\BCC_1(\BCC_2( \cdots (\BCC_n (P)) \cdots )) = P_I.
\]
This establishes that for every polytope $P$, one can obtain its integer hull with at most $n$ applications of the $\BCC$ operator.

While iteratively applying $\BCC$ in all $n$ indices is intractable (unless $\mathcal{P} = \mathcal{NP}$), applying them simultaneously to $P$ and intersecting them is not. Furthermore, it is easy to see that $P_I$ is contained in the intersection of these $n$ sets. Thus,
\[
\LS_0(P) := \bigcap_{i \in [n]} \BCC_i(P),
\]
devised by Lov{\' a}sz and Schrijver~\cite{LovaszS91a}, is a relaxation of $P_I$ that is at least as tight as $\BCC_i(P)$ for all $i \in [n]$. \autoref{fig1} illustrates how $\BCC$ and $\LS_0$ operate in two dimensions.

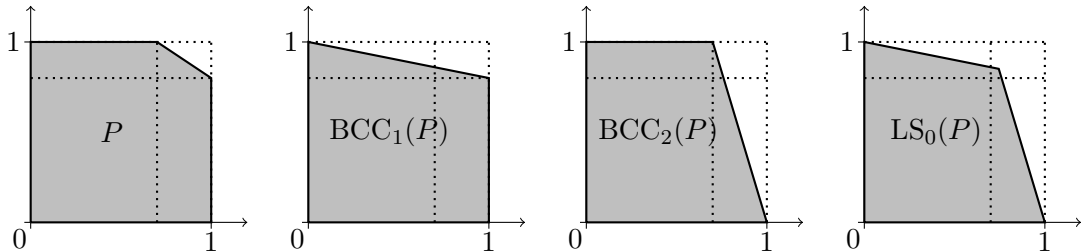
\begin{figure}[htb]

\begin{center}

\begin{tabular}{cccc}

\begin{tikzpicture}[scale = 0.8, y=3cm, x=3cm]
	\draw[->] (0,0) -- coordinate (x axis mid) (1.2,0);
    	\draw[->] (0,0) -- coordinate (y axis mid) (0,1.2);

\foreach \x in {0,...,1}
     		\draw (\x,1pt) -- (\x,-3pt);
    	\foreach \y in {0,...,1}
     		\draw (1pt,\y) -- (-3pt,\y) ;

\node[anchor=north] at (1,0) (label2) {1};
\node[anchor=east] at (0,1) (label3) {1};
\node at (-0.06,-0.09) {0};

\draw[fill = lightgray, thick] (0,0) -- (1,0) -- (1, 0.8) -- (0.7,1) -- (0,1) -- (0,0);
\draw[dotted, thick] (1,0) -- (1,1);
\draw[dotted, thick ] (0,1) -- (1,1);

\draw[dotted, thick] (0.7,0) -- (0.7,1);
\draw[dotted, thick] (0,0.8) -- (1,0.8);

\node at (0.45,0.5) (label) {$P$};
\end{tikzpicture}

&

\begin{tikzpicture}[scale = 0.8, y=3cm, x=3cm]
	\draw[->] (0,0) -- coordinate (x axis mid) (1.2,0);
    	\draw[->] (0,0) -- coordinate (y axis mid) (0,1.2);

\foreach \x in {0,...,1}
     		\draw (\x,1pt) -- (\x,-3pt);
    	\foreach \y in {0,...,1}
     		\draw (1pt,\y) -- (-3pt,\y) ;

\node[anchor=north] at (1,0) (label2) {1};
\node[anchor=east] at (0,1) (label3) {1};
\node at (-0.06,-0.09) {0};

\draw[fill = lightgray, thick] (0,0) -- (1,0) -- (1, 0.8) -- (0,1) -- (0,0);
\node at (0.45,0.5) (label) {$\BCC_1(P)$};
\draw[dotted, thick] (1,0) -- (1,1);
\draw[dotted, thick] (0,1) -- (1,1);
	
\draw[dotted, thick] (0.7,0) -- (0.7,1);
\draw[dotted, thick] (0,0.8) -- (1,0.8);

\end{tikzpicture}

&

\begin{tikzpicture}[scale = 0.8, y=3cm, x=3cm]
	\draw[->] (0,0) -- coordinate (x axis mid) (1.2,0);
    	\draw[->] (0,0) -- coordinate (y axis mid) (0,1.2);

\foreach \x in {0,...,1}
     		\draw (\x,1pt) -- (\x,-3pt);
    	\foreach \y in {0,...,1}
     		\draw (1pt,\y) -- (-3pt,\y) ;

\node[anchor=north] at (1,0) (label2) {1};
\node[anchor=east] at (0,1) (label3) {1};
\node at (-0.06,-0.09) {0};
	
\draw[fill = lightgray, thick] (0,0) -- (1,0) -- (0.7,1) -- (0,1) -- (0,0);
\node at (0.4,0.5) (label) {$\BCC_2(P)$};

\draw[dotted, thick] (1,0) -- (1,1);
\draw[dotted, thick] (0,1) -- (1,1);

\draw[dotted, thick] (0.7,0) -- (0.7,1);
\draw[dotted, thick] (0,0.8) -- (1,0.8);

\end{tikzpicture}

&

\begin{tikzpicture}[scale = 0.8, y=3cm, x=3cm]
	\draw[->] (0,0) -- coordinate (x axis mid) (1.2,0);
    	\draw[->] (0,0) -- coordinate (y axis mid) (0,1.2);

\foreach \x in {0,...,1}
     		\draw (\x,1pt) -- (\x,-3pt);
    	\foreach \y in {0,...,1}
     		\draw (1pt,\y) -- (-3pt,\y) ;

\node[anchor=north] at (1,0) (label2) {1};
\node[anchor=east] at (0,1) (label3) {1};
\node at (-0.06,-0.09) {0};

\draw[fill = lightgray, thick] (0,0) -- (1,0) -- (35/47, 40/47) -- (0,1) -- (0,0);
\node at (0.4,0.5) (label) {$\LS_0(P)$};

\draw[dotted, thick] (1,0) -- (1,1);
\draw[dotted, thick] (0,1) -- (1,1);

\draw[dotted, thick] (0.7,0) -- (0.7,1);
\draw[dotted, thick] (0,0.8) -- (1,0.8);

\end{tikzpicture}

\end{tabular}

\end{center}
\caption{An illustration of $\BCC$ and $\LS_0$ in two dimensions.}\label{fig1}
\end{figure}

Before we look into operators that are even stronger (and more sophisticated), it is helpful to understand the following alternative description of $\LS_0$. Given $x \in [0,1]^n$,
let $\het{x}$ denote the vector $\begin{pmatrix} 1 \\ x \end{pmatrix}$ in $\mathbb{R}^{n+1}$, where the new coordinate is indexed by zero. Let $e_i$ denote the $i^{\tn{th}}$ unit vector (of appropriate size), and for any square matrix $M$, let $\tn{diag}(M)$ denote the vector formed by the diagonal entries of $M$. Next, given $P \subseteq [0,1]^n$, define the cone
\[
K(P) := \set{\begin{pmatrix}\lambda \\ \lambda x \end{pmatrix} \in \mathbb{R}^{n+1}: \lambda \geq 0, x \in P}.
\]
Then, it is not hard to check that
\begin{eqnarray*}
\LS_0(P) &=& \left\{ x \in \mathbb{R}^n : \exists Y \in \mathbb{R}^{(n+1) \times (n+1)}, Ye_i,  Y(e_0 - e_i)\in K(P),~\forall i \in [n],\right.\\
&& \left. Ye_0 = Y^{\top}e_0 = \tn{diag}(Y) = \het{x} \right\}.
\end{eqnarray*}
To see that $\LS_0(P) \supseteq P_I$ in this perspective, observe that given any integral vector $x \in P$, the matrix $Y := \het{x} \het{x}^{\top}$ is a matrix which ``certifies''
that $x \in \LS_0(P)$. Then $P_I \subseteq \LS_0(P)$ follows from the fact that the latter is obviously a convex set.

Now, observe that $\het{x}\het{x}^{\top}$ is symmetric for all $x \in \set{0,1}^n$. Thus, if we let $\mathbb{S}^n$ denote the set of $n$-by-$n$ real, symmetric matrices, then
\[
\LS(P) := \set{x \in \mathbb{R}^n : \exists Y \in \mathbb{S}^{n+1}, Ye_i,  Y(e_0 - e_i)\in K(P),~\forall i \in [n], Ye_0 = \tn{diag}(Y) = \het{x}}
\]
also contains $P_I$. By enforcing a symmetry constraint on the matrices in the lifted space (and still retaining all integral points in $P$), we see that $\LS(P)$ is a potentially tighter relaxation than $\LS_0(P)$. We can also apply these operators iteratively to a polytope $P$ to gain progressively tighter relaxations. Let $\LS_0^k(P)$ (resp. $\LS^k(P)$) denote the set obtained from applying $\LS_0$ (resp. $\LS$) to $P$ iteratively for $k$ times. Since it is apparent from their definitions that
$\LS(P) \subseteq \LS_0(P) \subseteq \BCC_i(P)$,
for every $i \in [n]$, it follows that $\LS_0^n(P) = \LS^n(P) = P_I$, for every $P \subseteq [0,1]^n$.

In the two aforementioned Lov{\'a}sz--Schrijver operators, the certificate matrices all have dimension $(n+1)$ by $(n+1)$. We next look into the potential of lifting the initial relaxation $P \subseteq [0,1]^n$ to sets of even higher dimensions. From here on, we denote $\set{0,1}^n$ by $\F$, and define $\mathcal{A} := 2^{\F}$, the power set of $\F$. For each $x \in \F$, we define the vector $x^{\A} \in \mathbb{R}^{\A}$ where
\[
x^{\A}_{\alpha} = \left\{
\begin{array}{ll}
1 & \tn{if $x \in \alpha$;}\\
0 & \tn{otherwise.}
\end{array}
\right.
\]
That is, each coordinate of $\A$ corresponds to a subset of the vertices of the $n$-dimensional unit hypercube, and $x^{\A}_{\alpha} = 1$ if and only if the point $x$ is contained in the set $\a$.
It is not hard to see that for all $x \in \F$, we have $x^{\A}_{\F} = 1$, and
$x^{\A}_{\set{y \in \F : y_i=1}} = x_i, \forall i \in [n].$
Another important property of $x^{\A}$ is that, given disjoint subsets $\alpha_1, \alpha_2, \ldots, \alpha_k \subseteq \b \subseteq \F$, we know that
\begin{equation}\label{setpartition}
x^{\A}_{\alpha_1} + x^{\A}_{\alpha_2} + \cdots + x^{\A}_{\alpha_k} \leq x^{\A}_{\b},
\end{equation}
and equality holds if $\set{ \alpha_1, \alpha_2, \ldots, \alpha_k}$ partitions $\b$.

Thus, for any given $x \in \F$, if we define $Y_{\A}^x := x^{\A}(x^{\A})^{\top}$, then the entries of $Y^x_{\A}$ have considerable structure. Most notably, the following must hold:
\begin{itemize}
\item[(P1)]
$Y_{\A}^x e_{\F} = (Y_{\A}^x)^{\top} e_{\F} = \tn{diag}(Y_{\A}^x) = x^{\A}$;
\item[(P2)]
$Y_{\A}^x e_{\alpha} \in \set{0,x^{\A}},~\forall \alpha \in {\A}$;
\item[(P3)]
$Y_{\A}^x \in \mathbb{S}^{\A}$;
\item[(P4)]
$Y_{\A}^x[\alpha,\beta] = 1 \iff x \in \alpha \cap \beta$;
\item[(P5)]
if $\a_1 \cap \b_1 = \a_2 \cap \b_2$, then $Y_{\A}^x[\a_1, \b_1] = Y_{\A}^x[\a_2, \b_2]$;
\item[(P6)]
every row and column of $Y_{\A}^x$ satisfies~\eqref{setpartition}.
\end{itemize}

Of course, $Y_{\A}^x$ has double-exponential size (in $n$), and explicitly constructing elements in a lifted space of such a high dimension could yield an intractable structure, which makes the underlying algorithm no better than simply enumerating the integral points in $P$. Nevertheless, we can try to obtain a tight relaxation by only working with polynomial-size submatrices of $Y_{\A}^x$, and imposing  constraints that are relaxations of the conditions (P1) to (P6), in hope of capturing some important inequalities that are valid for $P_I$ but not $P$. Zuckerberg~\cite{Zuckerberg03a} showed that most of the existing lift-and-project operators can be interpreted under this common theme.

We next express the operators devised by Sherali and Adams~\cite{SheraliA90a} in this language. Given a set of indices $S \subseteq [n]$ and $t \in \set{0,1}$, we define
\[
S |_t := \set{ x \in \F : x_i = t,~\forall i \in S}.
\]
Note that $\es|_0 = \es|_1 = \F$. Also, to reduce cluttering, we write $i|_t$ instead of $\set{i}|_t$. Next, given any integer $\l \in \set{0,1,\ldots,n}$, we define $\A_{\l} := \set{S|_1 \cap T|_0 : S,T \subseteq [n], S \cap T = \es, |S| + |T| \leq \l}$ and $\A_{\l}^+ := \set{S|_1 : S \subseteq [n],  |S| \leq \l}$. For instance,
\[
\A_{1} = \set{\F, 1|_1, 2|_1, \ldots, n|_1, 1|_0, 2|_0, \ldots, n|_0},
\]
while
\[
\A_{1}^+ = \set{\F, 1|_1, 2|_1, \ldots, n|_1}.
\]
Also, given any vector $y \in \mathbb{R}^{\A'}$ for some $\A' \subseteq \A$ which contains $\F$ and $i|_1$ for all $i \in [n]$, we let $\het{x}(y) := (y_{\F}, y_{1|_1}, \ldots, y_{n|_1})^{\top}$. To relate these vectors with the  $\het{x}$ vectors defined previously, sometimes we may also alternatively index the entries of $\het{x}(y)$ as $(y_0, y_1, \ldots, y_n)^{\top}$.

For any fixed integer $k \in [n]$, the $\SA^k$ operator can be defined as follows:

\begin{enumerate}
\item
\noindent
Let $\tilde{\SA}^k(P)$ denote the set of matrices $Y \in \mathbb{R}^{\A^+_1 \times \A_k}$ which satisfy all of the following conditions:
\begin{itemize}
\item[($\SA 1$)]
$Y[\F, \F] = 1$.
\item[($\SA 2$)]
$Ye_{\a} \in K(P)$, for every $\a \in \A_k$.
\item[($\SA 3$)]
For every $S|_1 \cap T|_0 \in \A_{k-1}$,
\[
Ye_{S|_1 \cap T|_0 \cap j|_1} + Ye_{S|_1 \cap T|_0 \cap j|_0} = Ye_{S|_1 \cap T|_0}, \quad \forall j \in [n] \sm (S \cup T).
\]
\item[($\SA 4$)]
For all $\a \in \A^+_1,  \b \in \A_k$ such that $\a \cap \b = \es$, $Y[\a,\b] = 0$.
\item[($\SA 5$)]
For all $\a_1,\a_2 \in \A^+_1, \b_1, \b_2 \in \A_k$ such that $\a_1 \cap \b_1 = \a_2 \cap \b_2$, $Y[\a_1, \b_1] = Y[\a_2, \b_2]$.
\end{itemize}
\item
Define
\[
\SA^k(P) := \set{x \in \mathbb{R}^n: \exists Y \in \tilde{\SA}^k(P),  Y e_{\F}= \het{x}}.
\]
\end{enumerate}

The $\SA^k$ operator was originally described by linearizing polynomial inequalities, as follows: given an inequality $\sum_{i=1}^n a_i x_i \leq a_0$ that is valid for $P$, disjoint subsets of indices $S,T \subseteq [n]$ such that $|S| + |T| \leq k$, $\SA^k$ generates the inequality
\begin{equation}\label{SAoriginal}
\left( \prod_{i\in S} x_i\right) \left( \prod_{i \in T} (1- x_i) \right) \left(\sum_{i=1}^n a_i x_i \right) \leq \left( \prod_{i\in S} x_i\right) \left( \prod_{i \in T} (1- x_i) \right)a_0,
\end{equation}
and obtains a linear inequality by replacing the monomial $x_i^j$ with $x_i$ (for all $j \geq 2$) in all terms, and then by using a new variable to represent each nontrivial product of monomials. In our definition of $\SA^k$, the linearized inequality would be
\[
\sum_{i=1}^n a_i Y[i|_1, S|_1 \cap T|_0] \leq a_0 Y[\F, S|_1 \cap T|_0],
\]
which is enforced by $(\SA 2)$ on the column of $Y$ indexed by the set $S|_1 \cap T|_0$. Also, for any set of indices $U \subseteq [n]$, the product of monomials $\prod_{i \in U} x_i$ could appear multiple times in the original formulation when we generate~\eqref{SAoriginal} using different $S$ and $T$. Then $\SA^k$ identifies them all by the variable $x_U$ in the linearized formulation. This requirement is enforced by ($\SA 5$) in our definition. Also observe that ($\SA 3$) ensures the matrix entries representing the products $x_j \left( \prod_{i\in S} x_i\right) \prod_{i \in T} \left(1- x_i \right)$ and
$(1-x_j)\left(\prod_{i\in S} x_i\right) \prod_{i \in T} \left(1- x_i \right)$ do sum up to that representing $\left( \prod_{i\in S} x_i\right) \prod_{i \in T} \left( 1- x_i \right)$. Finally, notice that the monomial $x_j(1-x_j) = x_j - x_j^2$ vanishes after linearizing. Thus, if $S,T$ are not disjoint, the product $\left( \prod_{i\in S} x_i\right) \prod_{i \in T} \left(1- x_i \right)$ vanishes, and ($\SA 4$) enforces that the corresponding matrix entries take on value zero.

It is not hard to see that $\SA^1(P) = \LS(P)$. In general, $\SA$ obtains extra strength over $\LS$ by lifting $P$ to a set of matrices of higher dimension, and using some properties of sets in $\A$ to identify variables in the lifted space.  For a comparison of $\SA$ and $\LS$, see Laurent~\cite{Laurent03a}.

Finally, we look into the polyhedral lift-and-project operator devised by Bienstock and Zuckerberg~\cite{BienstockZ04a}. Recall that the idea of convexification requires a collection of disjoint subsets of $P$ whose union contains all integral points in $P$. So far, every operator that we have seen obtains these sets by intersecting $P$ with faces of $[0,1]^n$. However, sometimes it is beneficial to allow more flexibility in choosing the way we partition the integral points in $P$. For example, consider
\[
P := \set{ x \in [0,1]^n : \sum_{i=1}^n x_i \leq  n- \frac{1}{2}}.
\]
In this case, $\SA^{n-1}(P)$, a relaxation obtained from using convexification with exponentially many sets that are all intersections of $P$ and faces of $[0,1]^n$, still strictly contains $P_I$. On the other hand, if we define
\[
Q_j := \set{ x \in P : \sum_{i=1}^n x_i = j},
\]
for every $j \in \set{0,1,\ldots, n}$, then every integral point in $P$ is contained in $Q_j$ for some $j$, and
\[
P_I = \tn{conv}\left(\bigcup_{i=0}^n Q_j \right).
\]
We will see in the next section that, given any set $P \subseteq [0,1]^n$, the set $\tn{conv}\left(\bigcup_{i=0}^n Q_j \right)$ can be described as the projection of a set of dimension $O(n^2)$ that is tractable as long as $P$ is.

Bienstock and Zuckerberg~\cite{BienstockZ04a} utilized this type of ideas and invented operators that use variables in $\A$ that were not exploited by the operators proposed earlier, in conjunction with some new constraints. We will denote their polyhedral operator by $\BZ$, but we also present variants of it called $\BZ'$ and $\BZ''$. These modified operators have the advantage of being stronger, and are also simpler to present. Moreover, since we are mostly interested in applying these operators to polytopes that arise from set packing problems (such as the stable set and matching problems of graphs), we will state versions of these operators that only apply to lower-comprehensive polytopes. We will discuss this in more detail after stating the elements of their operators.

Suppose we are given a polytope $P := \set{x \in [0,1]^n : Ax \leq b}$, where $A \in \mathbb{R}^{m \times n}$ is nonnegative and $b \in \mathbb{R}^m$ is positive (this implies that $P$ is lower-comprehensive). The $\BZ'$ operator can be viewed as a two-step process. The first step is \emph{refinement}.
Given a vector $v$, let $\supp(v)$ denote the \emph{support} of $v$. Also, for every $i \in [m]$, let $A^i$ denote the $i^{\tn{th}}$ row of $A$. If $O \subseteq [n]$ satisfies
\begin{itemize}
\item
$O \subseteq \supp(A^i)$;
\item
$\sum_{j \in O} A^i_j > b_i$; and
\item
$|O| \leq k+1$ or $|O| \geq |\supp(A^i)| - (k+1)$
\end{itemize}
for some $i \in [m]$, then we call $O$ a \emph{$k$-small obstruction}. Let $\O_k$ denote the collection of all $k$-small obstructions of $P$ (or more precisely, of the system $Ax\leq b$). Notice that, for every obstruction $O \in \O_k$, and integral vector $x \in P$, the inequality $\sum_{i \in O} x_i \leq |O| - 1$ holds. Thus,
\[
\O_k(P) := \set{x \in P: \sum_{i \in O} x_i \leq |O|-1,~\forall O \in \O_k}
\]
is a relaxation of $P_I$ that is potentially tighter than $P$. 

The second step of the $\BZ'^k$ operator is \emph{lifting}. Before we give the details of this step, we need another intermediate set of indices, called \emph{walls}. For every $k \geq 1$, we define
\[
\W_k := \set{ \bigcup_{i,j \in [\l],  i\neq j} (O_i \cap O_j) : O_1,\ldots, O_{\l} \in \O_k, 2 \leq \l \leq k+1} \cup \set{ \set{1}, \ldots, \set{n}}.
\]
That is, each subset of up to $(k+1)$ $k$-small obstructions generate a wall, which is the set of elements that appear in at least two of the given obstructions. We also ensure that the singleton sets of indices are walls. Next, we define the collection of \emph{tiers}
\[
\T_k := \set{S \subseteq [n]: \exists W_{i_{1}}, \ldots, W_{i_{k}} \in \W_k, S \subseteq \bigcup_{j=1}^{k} W_{i_{j}}}.
\]
That is, we define a set of indices $S$ to be a tier if there exist $k$ walls whose union contains $S$. Note that every subset of $[n]$ of size up to $k$ is a tier. Finally, given a set $U \subseteq [n]$ and a nonnegative integer $r$, we define
\[
U |_{< r} := \set{ x \in \F : \sum_{i \in U} x_i \leq r-1}.
\]
We shall see that the elements in $\A$ that are being generated by $\BZ'$ all take the form \\ $S|_1 \cap T|_0 \cap U|_{< r}$, where $S,T,U$ are disjoint sets of indices. 
Next, we describe the lifting step of $\BZ'^k$:

\settowidth{\mylen}{$(S \sm (T \cup U) )|_1 \cap T|_0 \cap U|_{<|U|- (k-|T|)}$,}
\begin{enumerate}
\item
Define $\A'$ to be the set consisting of the following. For each tier $S \in \T_k$, include:
\begin{align}
\llap{\textbullet\hspace{92pt}} \makebox[\mylen][c]{$(S \sm T)|_1 \cap T|_0$,}\nonumber  \end{align}

for all $T \subseteq S$ such that $|T| \leq k$;
\begin{align}
\llap{\textbullet\hspace{92pt}}  (S \sm (T \cup U) )|_1 \cap T|_0 \cap U|_{<|U|- (k-|T|)},\nonumber  \end{align}

for every $T,U \subseteq S$ such that $U \cap T = \es, |T| < k$ and $|U| + |T| > k$.
\vspace{2mm}

\noindent
We say these variables (indexed by the above sets) are associated with the tier $S$.
\item
Let $\tilde{\BZ}'^k(P)$ denote the set of matrices $Y \in \mathbb{S}^{\A'}$ that satisfy all of the following conditions:
\begin{itemize}
\item[($\BZ'1$)]
$Y[\F, \F] = 1$.
\item[($\BZ'2$)]
For every column $y$ of the matrix $Y$,
\begin{itemize}
\item[(i)]
$0 \leq y_{\alpha} \leq y_{\F}$, for all $\a \in \A'$.
\item[(ii)]
$\het{x}(y) \in K(\O_k(P))$.
\item[(iii)]
$y_{i|_1} + y_{i|_0} = y_{\F}$, for every $i \in [n]$.
\item[(iv)]
For each $\a \in \A'$ of the form of $S|_1 \cap T|_0$ impose the inequalities
\begin{eqnarray}
\label{Ck41} y_{i|_1} &\geq& y_{\a}, \quad \forall i \in S; \\
\label{Ck42} y_{i|_0} &\geq& y_{\a}, \quad \forall i \in T; \\
\label{sumwall1} y_{\a} + y_{(S \cup \set{i})|_1 \cap (T \sm \set{i})|_0}  &= & y_{S|_1 \cap (T \sm \set{i})|_0}, \quad \forall i \in T; \\
\label{Ck43} \sum_{i \in S} y_{i|_1} + \sum_{i \in T} y_{i|_0} - y_{\alpha} &\leq& (|S| + |T| -1)y_{\F}.
\end{eqnarray}
\item[(v)]
For each $\a \in \A'$ of the form $S|_1 \cap T|_0 \cap U|_{< r}$, impose the inequalities
\begin{eqnarray}
\label{Ck44} y_{i|_1} &\geq& y_{\a}, \quad \forall i \in S; \\
\label{Ck45} y_{i|_0} &\geq& y_{\a}, \quad \forall i \in T; \\
\label{Ck46} \sum_{i \in U} y_{i|_0} &\geq& (|U| - (r-1)) y_{\alpha};\\
\label{sumwall2} y_{\a} &=& y_{S|_1 \cap T|_0} - \sum_{U' \subseteq U, |U'| \geq r} y_{(S \cup U')|_1 \cap (T \cup (U \sm U'))|_0}.
\end{eqnarray}
\end{itemize}
\item[($\BZ'3$)]
For all $\a, \b \in \A'$ such that $\conv(\a) \cap \conv(\b) \cap P = \es$, $Y[\a,\b] = 0$.
\item[($\BZ'4$)]
For all $\a_1,\b_1, \a_2,\b_2 \in \A'$ such that $\a_1 \cap \b_1 = \a_2 \cap \b_2$, $Y[\a_1, \b_1] = Y[\a_2, \b_2]$.
\end{itemize}
\item
Define
\[
\BZ'^k(P) := \set{x \in \mathbb{R}^n: \exists Y \in \tilde{\BZ}'^k(P),  \het{x}(Ye_{\F})= \het{x}}.
\]
\end{enumerate}

Similar to the case of $\SA^k$, $\BZ'^k$ can be seen as creating columns that correspond to sets that partition $\F$. While $\SA^k$ only generates a partition for each subset of up to $k$ indices, $\BZ'^k$ does so for every tier, which is a much broader collection of indices. For a tier $S$ up to size $k$, it does the same as $\SA^k$ and generates $2^{|S|}$ columns corresponding to all possible complementations of indices in $S$. However, for $S$ of size greater than $k$, it generates a column for $(S \sm T)|_1 \cap T|_0$ for each $T \subseteq S$ of size up to $k$, and a column for $S|_{< |S| - k}$. This can be seen intuitively as a ``$k$-deep'' partition of $\F$ corresponding to $S$ --- sets that can be obtained from starting with $S|_1$ and complementing no more than $k$ entries in $S$ are each represented by a matrix column in the lifted space, while all other sets that are more than $k$ complementations away from $S|_1$ is represented by a single column in the matrix. For example, suppose $\BZ'^1$ is applied to a polytope and $S = \set{1,2,3}$ is a tier. Then the algorithm would generate columns corresponding to the sets
\[
\set{1,2,3}|_1,~~\set{2,3}|_1 \cap \set{1}|_0,~~\set{1,3}|_1 \cap \set{2}|_0,~~\set{1,2}|_1 \cap \set{3}|_0,~~\set{1,2,3}|_{<3}.
\]
Note that the five sets given above partition $\F$. In fact, given a tier $S$ and $T \subseteq S$ such that $|T| < k$, $\BZ'^k$ also generates a $(k - |T|)$-deep partition of this set for each $U \subseteq S \sm T$ such that $|U| + |T| > k$. First, the column for
\[
(S \sm (T \cup U'))|_1 \cap (T \cup U')|_0 
\]
is generated for all $U' \subseteq U$ of size $\leq k - |T|$ (i.e. if the set is at no more than $k- |T|$ complementations away from $(S \sm (T \cup U))|_1 \cap T|_0$). Then $\BZ'^k$ also generates
\[
(S \sm (T \cup U))|_1 \cap T|_0 \cap U|_{< |U|- (k - |T|)}
\]
to capture the remainder of the partition. 

Since each singleton index set is a wall, we see that every index set of size up to $k$ is a tier. Thus, $\A'$ contains $\A_{k}$, and it is not hard to see that $\BZ'^{k}(P) \subseteq \SA^{k}(\O_k(P))$ in general. ($\BZ'^k$ also dominates $\SA'^k$, a stronger version of $\SA^k$ that will be defined after the next theorem.) Furthermore, notice that in $\BZ'$, we have generated exponentially many variables, whereas in the original $\BZ$ only polynomially many are selected. The role of walls is also much more important in selecting the variables in $\BZ$, which we have intentionally suppressed in $\BZ'$ to make our presentation and analysis more transparent. Most of our lower-bound results are established on the stronger operator $\BZ'$, which implies that similar lower-bound results hold for all operators dominated by $\BZ'$, such as $\BZ$ and $\SA$. Some of the details of the relationships between these modified operators and the original Bienstock--Zuckerberg operators are given in the Appendix.

While Bienstock and Zuckerberg's original definition of $\BZ^k$ accepts any polytope as input, they showed that their operator works particularly well on certain instances of set covering problems. One of their main results is the following: Given an inequality $a^{\top}x \geq a_0$ such that $a \geq 0$ and $a_0 >0$, its \emph{pitch} is defined to be the smallest positive integer $j$ such that
\[
S \subseteq \tn{supp}(a), |S| \geq j \Rightarrow \sum_{i \in S} a_i \geq a_0.
\]
Let $\bar{e}$ denote the all-ones vector of suitable size. Then Bienstock and Zuckerberg showed the following powerful result:

\begin{tm}[Bienstock and Zuckerberg~\cite{BienstockZ04a}]\label{BZpitchk}
Suppose $P := \set{x \in [0,1]^n : Ax \geq \bar{e}}$ where $A$ is a $0,1$ matrix. Then for every $k\geq 1$, every valid inequality of $P_I$ that has pitch at most $k+1$ is valid for $\BZ^k(P)$.
\end{tm}

Note that if all coefficients of an inequality are integral and at most $k$, then the pitch of the inequality is no more than $k$.

One major distinction between the Bienstock--Zuckerberg operators and the earlier ones is that they may generate different variables for different input set $P$. In fact, the performance of $\BZ$ can vary upon different algebraic descriptions of the given set $P$, even if they geometrically describe the same set. For instance, adding a redundant inequality to the system $Ax \leq b$ could make many more sets qualify as $k$-small obstructions. This could increase the dimension of the lifted set as more walls and tiers are generated, and as a result possibly strengthen the operator. We provide examples that illustrate this phenomenon in the Appendix.

Next, we take a closer look into the condition $(\BZ' 3)$, which is one of the conditions used in the Bienstock--Zuckerberg operators that were not explicitly imposed by the earlier lift-and-project operators. Observe that, for
every $x \in P \cap \set{0,1}^n$,
\[
Y_{\A}^x[\a,\b] = x^{\A}_{\a} x^{\A}_{\b} = 0
\]
whenever $\a \cap \b \cap P = \es$. Thus, imposing $Y[\a, \b] = 0$ whenever $\conv(\a) \cap \conv(\b) \cap P = \es$ still preserves all matrices in the lifted space which correspond to integral points in $P$. Also, note that this condition can be efficiently checked for the variables that may be selected in $\BZ'$. For instance, for $\a = S|_1 \cap T|_0 \cap U|_{< r}$,
\[
\conv(\a) = \set{ x \in [0,1]^n : x_i = 1, \forall i \in S, x_i = 0, \forall i \in T, \sum_{i \in U} x_i \leq r-1}.
\]
Thus, checking if $x \in \conv(\a) \cap \conv(\b) \cap P$ for any specific pair of $\a,\b$ amounts to verifying if $x$ satisfies $O(n)$ linear equations and inequalities (in addition to verifying
membership in $P$), which is tractable.

Since we will relate the performance of $\BZ'$ and $\BZ'_+$ to other operators (such as $\SA$), it is worthwhile to investigate how this new condition impacts the overall strength of an operator. Given $P \subseteq [0,1]^n$, and integer $k \geq 1$, define
\[
\SA'^{k}(P) := \set{x \in \mathbb{R}^n: \exists Y \in \tilde{\SA}'^{k}(P):  Y e_{\F}= \het{x}},
\]
where $\tilde{\SA}'^{k}(P)$ is the set of matrices in $\tilde{\SA}^k(P)$ that satisfy
\begin{itemize}
\item[($\SA' 4$)]
For all $\a \in \A^+_1,  \b \in \A_k$ such that $\conv(\a) \cap \conv(\b) \cap P= \es$, $Y[\a,\b] = 0$.
\end{itemize}

Note that $\SA'^k$ yields a tractable algorithm when $k = O(1)$, since the condition ($\SA' 4$) --- as with ($\BZ' 3$), as explained above --- can be verified efficiently (assuming $P$ is tractable), and is only checked polynomially many times. Also, since $(\SA' 4)$ is more restrictive than $(\SA 4)$, it is apparent that $\SA'^{k}(P) \subseteq \SA^k(P)$ for every set $P \subseteq [0,1]^n$. However, it turns out that in the case of $\SA$, this extra condition would ``save'' at most one iteration.

\begin{prop}\label{SASA'}
For every $P \subseteq [0,1]^n$ and every $k\geq 1$,
\[
\SA^{k+1}(P) \subseteq \SA'^{k}(P).
\]
\end{prop}

\begin{proof}
Let $x \in \SA^{k+1}(P)$, and let $Y \in \tilde{\SA}^{k+1}(P)$ such that $Y e_{\F} = \het{x}$. Define $Y' \in \mathbb{R}^{\A_1^+ \times \A_k}$ such that $Y'[\a,\b] = Y[\a,\b],~\forall \a \in \A_1^+, \b \in \A_k$ (i.e., $Y'$ is a submatrix of $Y$). Since $Y' e_{\F} = Ye_{\F} = \het{x}$, it suffices to show that $Y' \in \tilde{\SA}'^{k}(P)$.

By construction, it is obvious that $Y' \in \tilde{\SA}^k(P)$. Thus, we just need to show that $Y'$ satisfies $(\SA' 4)$. Given $\a \in \A_1^+, \b \in \A_k$, suppose $\a = i|_1$, and $\b = S|_1 \cap T|_0$ for $S,T \subseteq [n]$. Now $\a \cap \b = (S \cup \set{i})|_1  \cap T|_0 \in \A_{k+1}$, and thus the entry $Y[\F, \a \cap \b]$ exists.

Since $Y e_{\a \cap \b} \in K(P)$ by ($\SA 2$), $Y[\F, \a \cap \b]  > 0$ would imply that the point
\[
y := \frac{1}{Y[\F, \a \cap \b]} (Y[1|_1, \a \cap \b], Y[2|_1, \a \cap \b], \ldots, Y[n|_1, \a \cap \b])^{\top}
\]
is in $P$. By ($\SA 5$), we know that $Y[j|_1, \a \cap \b] = Y[\F, \a \cap \b]$ if $j \in S \cup \set{i}$, and by ($\SA 3$), we have $Y[j|_1, \a \cap \b] = 0$ if $j \in T$. Thus, it follows that $y$ belongs to $\conv(\a)$ and $\conv(\b)$. Therefore,  $(\SA' 4)$ holds as $\conv(\a) \cap \conv(\b) \cap P \neq \es$, and our claim follows.
\end{proof}

\autoref{SASA'} establishes the dashed arrow from $\SA'$ to $\SA$ in \autoref{fig2}, and assures that if one can provide a performance guarantee for $\SA'$ on a polytope $P$, then the same can be said of the \emph{weaker} $\SA$ operator by using one extra iteration. The meanings for the other four dashed arrows in \autoref{fig2} are similar in nature  --- for some linear or quadratic function of the iterate number, the weaker operator can be at least as strong as the stronger operator. However, they are much more involved than \autoref{SASA'}, and sometimes depend on the properties of the given set $P$. We will address them in detail in the subsequent sections.

\section{Identifying Unhelpful Variables in the Lifted Space}

As we have seen in the previous section, one way to gain additional strength in devising a lift-and-project operator is to lift to a space of higher dimension, and obtain a potentially tighter formulation by using more variables
(and new constraints), albeit at a computational cost. In this section, we provide conditions on sets and higher dimensional liftings which do not lead to strong cuts. As a result, we show in some cases, $\BZ'^k$ performs no better than $\SA'^{\l}$ for some suitably chosen pair $k$ and $\l$.

\subsection{A General Template}

\ignore{We first establish a few useful definitions.} Recall that $\F = \set{0,1}^n$, and $\A$ is the power set of $\F$. A common theme among all lift-and-project operators we have looked at so far is that their lifted spaces can all be interpreted as sets of matrices whose columns and rows are indexed by elements in $\A$.
Moreover, they all impose a constraint in the tune of ``each column of the matrix belongs to a certain set linked to $P$'' (e.g. conditions ($\SA 2$) and ($\BZ' 2$)). This provides a natural way of partitioning the constraints of a lift-and-project operator into two categories: those that are present (and identical) for every matrix column, and the remaining constraints that cannot be captured this way.

Let $\Gamma$ be a lift-and-project operator which lifts a given set $P$ to $\tilde{\Gamma}(P)$, and then projects it back onto the space where $P$ lives,
resulting in the output relaxation
$\Gamma(P)$. We say that $\Gamma$ is \emph{admissible} if it possesses all of the following properties:

\begin{itemize}
\item[(I1)]
Given a convex set $P \subseteq [0,1]^n$, $\Gamma$ lifts $P$ to a set of matrices $\tilde{\Gamma}(P) \subseteq \mathbb{R}^{\S \times \S'}$, such that
\[
\A_1^+ \subseteq \S \subseteq \S' \subseteq \A.
\]
\item[(I2)]
There exist a \emph{column constraint function} $f$ that maps elements in $\A$ to subsets of $\mathbb{R}^{\S}$, and a \emph{cross-column constraint function} $g$ that maps sets contained in $[0,1]^n$ to sets of matrices in $\mathbb{R}^{\S \times \S'}$, such that
\[
\tilde{\Gamma}(P) = \set{Y \in g(P) : Y e_{S'} \in f(S'),~\forall S' \in \S'}.
\]
Furthermore, $f$ has the property that, for every pair of disjoint sets $S, T \in \S'$:
\begin{enumerate}
\item
$f(S) \cup f(T) \subseteq f(S \cup T)$;
\item
$f(S) = f(T)$ if $S \cap P = T \cap P$.
\end{enumerate}
\item[(I3)]
\[
\Gamma(P) := \set{x \in \mathbb{R}^n : \exists Y \in \tilde{\Gamma}(P), Y[\F,\F]=1, \het{x}(Y e_{\F}) = \het{x}}.
\]
\end{itemize}

Loosely speaking, an admissible operator returns a relaxation $\Gamma(P)$ that is a projection of some set of matrices $\tilde{\Gamma}(P)$ whose rows and columns are indexed by entries in $\A$, with some structures that are captured by the functions $f$ and $g$. As we will see in subsequent results, the intention of the definition is to try to capture as much of $\Gamma$ as possible with $f$ by using it to describe the constraints $\Gamma$ places on every column of the matrices in the lifted space, and only include the remaining constraints in $g$. Thus, we want $f$ to be maximal, and $g$ to be minimal in this sense. For instance, we can show that $\SA^k$ is admissible by defining $f(S) := K(P \cap \tn{conv}(\a)),~\forall \a \in \A$ and $g(P)$ to be the set of matrices in $\mathbb{R}^{\A_1^+ \times \A_k}$ that satisfy ($\SA 3$), ($\SA 4$) and ($\SA 5$). All named operators mentioned in this manuscript can be shown to be admissible in this fashion --- using $f$ to describe that each matrix column has to be in some lifted set determined by $P$, and letting $g$ capture the remaining constraints. On the other hand, for any lift-and-project operator $\Gamma$ that satisfies (I1), we can show that it is admissible by letting $g(P) := \tilde{\Gamma}(P)$ and $f(\a) := \mathbb{R}^{\S}$ for all $\a \in \A$ (i.e., we define $f$ to be trivial and ``shove'' all constraints of $\Gamma$ under $g$). Thus, the notion of admissible operators is extremely broad, and the framework that we present here might also be applicable to the analyses of future lift-and-project operators that are drastically different from the existing ones.

For many known operators, these ``other'' constraints placed by $g$ are relaxations of the set theoretical properties (P5) and (P6) of $Y_{\A}^x$. For instance, ($\SA 5$) is in place to make sure the variables in the linearized polynomial inequalities that would be identified in the original description of $\SA^k$ would in fact have the same value in all matrices in $\tilde{\SA}^k(P)$. Likewise, ($\SA 3$) and ($\SA 4$) are also needed to capture the relationship between the variables that would be established naturally in the original description with polynomial inequalities.

Furthermore, sometimes using matrices to describe the lifted space and assigning set theoretical meanings to their columns and rows has advantages over using linearized polynomial inequalities directly. For instance, we again consider the set
\[
P := \set{ x \in [0,1]^n : \sum_{i=1}^n x_i \leq  n- \frac{1}{2}}.
\]
We have seen that if we define
\[
Q_j := \set{ x \in P : \sum_{i=1}^n x_i = j},
\]
for every $j \in \set{0,1,\ldots, n}$, then $P_I = \tn{conv} \left(\bigcup_{j=0}^n Q_j \right)$. However, if we attempt to construct a formulation by linearizing polynomial inequalities as in the original description of $\SA$, then to capture the constraints for $Q_j$ one would need to linearize
\[
\sum_{\substack{S,T : S\cup T = [n],\\ S \cap T = \es , |S| = j}} \left( \prod_{i\in S} x_i\right) \left( \prod_{i \in T} (1-x_i) \right) \left( \sum_{i=1}^n a_ix_i \right) \leq
\sum_{\substack{S,T : S\cup T = [n],\\ S \cap T = \es , |S| = j}}\left( \prod_{i\in S} x_i\right) \left( \prod_{i \in T} (1-x_i) \right) a_0
\]
for all inequalities $\sum_{i=1}^n a_ix_i \leq a_0$ that are valid for $P$. Of course, when $j \approx \frac{n}{2}$, the above constraint would have exponentially many terms.

However, we can obtain an efficient lifted formulation by doing the following: for each $j \in \set{0,1,\ldots,n}$, define $R_j \in \A$ where
\[
R_j = \set{ x \in \F: \sum_{i=1}^n x_i = j},
\]
and let $\S = \set{\F, R_0, R_1, \ldots, R_n}$. We now define $\Gamma$ to be the lift-and-project operator as follows:
\begin{enumerate}
\item
Given $P \subseteq [0,1]^n$, let $\tilde{\Gamma}(P)$ denote the set of matrices $Y \in \mathbb{R}^{\A_1^+ \times \S}$ such that
\begin{itemize}
\item[(i)]
$Y[\F,\F] = 1$.
\item[(ii)]
$Ye_{R_j} \in K(P \cap \tn{conv}(R_j)),~\forall j \in \set{0, \ldots, n}$.
\item[(iii)]
$Y e_{\F} = \sum_{j=0}^n Y e_{R_j}$.
\end{itemize}
\item
Define
\[
\Gamma(P) := \set{x \in \mathbb{R}^n: \exists Y \in \tilde{\Gamma}(P),  Y e_{\F}= \het{x}}.
\]
\end{enumerate}

Then it is not hard to see that $\Gamma(P) = \tn{conv}\left(\bigcup_{i=0}^n Q_j \right)$ for every set $P \subseteq [0,1]^n$. Note that we used constraint (iii) to enforce that the entries in the matrix behave consistently with their corresponding set theoretical meanings --- since $R_0, \ldots, R_n$ partition $\F$, we require that the columns indexed by the sets $R_0,\ldots, R_n$ sum up to that representing $\F$.

Thus, the following notions are helpful when we attempt to analyze cross-column constraint functions $g$ more systematically. First, given $\S, \S' \subseteq \A$, we say that $\S'$ \emph{refines} $\S$ if for all $S \in \S$, there exist mutually disjoint sets in $\S'$ that partition $S$. Equivalently, given $\S \subseteq \A$, let $Y_{\S}^x$ denote the $\A \times \S$ submatrix of $Y_{\A}^x$ consisting of the columns indexed by sets in $\S$. Then $\S'$ refines $\S$ if and only if every column $Y_{\S}^x$ is contained in the cone generated by the column vectors of $Y_{\S'}^x$, for every $x \in \F$. For instance, $\S'$ refines $\S$ whenever $\S \subseteq \S'$ (and thus $\A_{k}$ refines $\A_{k}^+$ for all $k \geq 0$, and $\A_{k}$ refines $\A_{\l}$ whenever $k \geq \l$). Note that the notion of refinement is transitive --- if $\S''$ refines $\S'$ and $\S'$ refines $\S$, then $\S''$ refines $\S$.

Next, given $Y_1 \in \mathbb{R}^{\S_1 \times \S'_1}$ and $Y_2 \in \mathbb{R}^{\S_2 \times \S'_2}$ where $\S_1,\S'_1, \S_2, \S'_2 \subseteq \A$, we say that $Y_1$ and $Y_2$ are \emph{consistent} if, given collections of mutually disjoint sets $\set{ S_{1i} \cap S'_{1i} : i \in [k]}$ and $\set{ S_{2i} \cap S'_{2i} : i \in [\l]}$,
\[
\bigcup_{i=1}^k \left( S_{1i} \cap S'_{1i} \right) = \bigcup_{i=1}^{\l} \left( S_{2i} \cap S'_{2i} \right) \Rightarrow \sum_{i=1}^k Y_1[S_{1i}, S'_{1i}] = \sum_{i=1}^{\l} Y_2[S_{2i}, S'_{2i}].
\]
Also, given a vector $y \in \mathbb{R}^{\S}$ where $\S \subseteq \A$, we can think of it as a $|\S|$-by-$1$
matrix whose single column is indexed by $\F$. Then we can extend the above notion to define
whether two vectors are consistent with each other, and whether a matrix and a vector are consistent
with each other.  For example, consider
\[
Y := \begin{pmatrix}
Y[\F, \F] & Y[\F, 1|_1] & Y[\F, 2|_1] \\
Y[1|_1, \F]& Y[1|_1, 1|_1] & Y[ 1|_1, 2|_1] \\
Y[ 2|_1, \F] & Y[2|_1, 1|_1] & Y[2|_1, 2|_1] 
\end{pmatrix} =
\begin{pmatrix}
1 & 0.7 & 0.4 \\ 
0.7 & 0.7 & 0.2 \\
0.4 & 0.2 & 0.4 
\end{pmatrix},
\]
and
\[
y := (y[\set{1,2}|_1], y[1|_1 \cap 2|_0], y[2|_1 \cap 1|_0], y[\set{1,2}|_0])^{\top} = (0.2,0.5,0.2,0.1)^{\top}.
\]
Then $Y$ is consistent with $y$. For example, notice that
\[
(1|_1 \cap  2|_0)  \cup (\set{1,2}|_1 ) = 1|_1.
\]
Accordingly, the corresponding entries in $Y$ and $y$ satisfy
\[
y[1|_1 \cap 2|_0] + y[ \set{1,2}|_1] = Y[\F, 1|_1] = 0.7.
\]
We remark that our notion of consistency is closely related to some similar notions used by Zuckerberg~\cite{Zuckerberg03a}.

Next, we say that a matrix $Y \in \mathbb{R}^{\S \times \S'}$, where $\S, \S' \subseteq \A$, is \emph{overall measure consistent} (OMC) if it is consistent with itself. All matrices in the lifted spaces of $\SA^k, \SA'^k$ and $\BZ'^k$ satisfy (OMC), for all $k \geq 1$.  For instance, a matrix $Y$ in $\tilde{\SA}^1(P)$ where $P \subseteq [0,1]^n$ takes the form
\[
Y = \begin{pmatrix}
Y[\F, \F] & Y[\F, 1|_1] & \cdots & Y[\F, n|_1] & Y[\F, 1|_0] & \cdots & Y[\F, n|_0]  \\
Y[ 1|_1, \F] & Y[ 1|_1, 1|_1] & \cdots &  Y[ 1|_1, n|_1]  & Y[ 1|_1, 1|_0] & \cdots &  Y[ 1|_1, n|_0]  \\
\vdots & \vdots & \ddots & \vdots & \vdots & \ddots & \vdots \\
Y[ n|_1, \F] & Y[n|_1, 1|_1] & \cdots & Y[n|_1, n|_1] & Y[n|_1, 1|_0] & \cdots & Y[n|_1, n|_0]
\end{pmatrix}.
\]
Then ($\SA 3$) enforces consistencies such as
\[
Y[i|_1, \F] = Y[i|_1, j|_1] + Y[i|_1, j|_0]
\]
for all indices $i,j$, while equations such as
\[
Y[\F, i|_1] = Y[i|_1, i|_1] + Y[i|_1, i|_0]
\]
follow from ($\SA 4$) (which enforces $Y[i|_1,i|_0] = 0$ as $i|_1 \cap i|_0 = \es$) and ($\SA 5$) (which enforces $Y[\F, i|_1] = Y[i|_1, i|_1]$ as $\F \cap i|_1 = i|_1 \cap i|_1 = i|_1$). One notable observation is the following: Suppose $x \in \mathbb{R}^{\S'}$ satisfies (OMC) and $\S'$ refines $\S$. Now for every $\a \in \S$, define $I_{\a} \subseteq \S'$ to be a collection of disjoint sets in $\S'$ that partitions $\a$.  Then, if we define $y \in \mathbb{R}^{\S}$ where
\[
y[\a] := \sum_{\b \in I_{\a}} x[\b]
\]
for every $\a \in \S$, then $y$ is the unique vector in $\mathbb{R}^{\S}$ that is consistent with $x$.

Finally, we are ready to formally describe some variables that we will show are unhelpful in the lifted space under this framework. Given an admissible operator $\Gamma$ and $P \subseteq [0,1]^n$, suppose $\tilde{\Gamma}(P) \subseteq \mathbb{R}^{\S \times \S'}$. If $\T = \set{T_1, \ldots, T_k} \subseteq \S'$ is a collection of sets where
\begin{enumerate}
\item
the set $\bigcup_{i=1}^k T_i$ is itself an element in $\left( \S' \sm \T \right)$; and
\item
there exists a unique $\l \in [k]$ such that $P \cap \tn{conv}(T_j) \neq \es$.
\end{enumerate}
Then we say that the sets $T_1, \ldots, T_k$ are \emph{$P$-useless}.

What does it mean for variables to be $P$-useless? For example, consider $\SA^2$ applied to a set $P \subseteq [0,1]^3$ in which no point satisfies
$x_3=1$. Then let $\T = \set{T_1, T_2} = \set{ \set{1,3}|_1, 1|_1 \cap 3|_0}$. Now consider any matrix $Y \in \tilde{\SA}^2(P)$. Since $P \cap T_1 = \es$, $Y e_{T_1} \in K(P)$ (enforced by ($\SA 2$)) implies that the entire column of $Y$ indexed by $T_1$ is zero. Next, let $R :=  T_1 \cup T_2 = 1|_1$, which is itself a variable generated by $\SA^2$. By ($\SA 3$), we know that $Ye_{T_1} + Ye_{T_2} = Ye_{R}$. Since we just argued that $Ye_{T_1}$ is the zero vector, we obtain that $Ye_{T_2} = Ye_{R}$ for all matrices $Y \in \tilde{\SA}^2(P)$. Since the column for $T_1$ is uniformly zero, and the column $T_2$ is redundant (it is identical to
the column for $R$), we can deem the variables $T_1, T_2$ $P$-useless, and not generate their columns when computing $\SA^2(P)$.

Geometrically, useless variables correspond to unfruitful partitions in the convexification process. Recall the idea that, given $P$ and $Q_i$'s are disjoint subsets of $P$ whose union contains all integral points in $P$, then the convex hull of these $Q_i$'s give a potentially tighter relaxation of $P_I$ than $P$. Now, if we have a set of indices $\T$ where the subcollection $\set{Q_i : i \in \T}$ has exactly one nonempty set, then we can replace that subcollection of $Q_i$'s by the single set $\bigcup_{i \in \T} Q_i$, and be assured that the convex hull of the reduced collection of $Q_i$'s would be the same as that of the original collection.

\ignore{ 
For example, let $T = \set{T_1,\ldots, T_k} \subseteq \S'$ be a collection of variables such that $R := \bigcup_{i=1}^k T_i$ is itself a variable in $\S'$, and $R \not\in T$. Further suppose there exists a unique $\l \in [k]$ such that $P \cap \tn{conv}(T_j) = \es,~\forall j \neq \l$. This means that in this setting, each of the $T_j$ has the set theoretical meaning of the empty set. Thus, we do not lose any points when projecting $\tilde{\Gamma}(P)$ to $\Gamma(P)$ if we assume that the matrix column indexed by $T_j$ is the vector of all zeros.  Moreover, since $R = \bigcup_{i=1}^k T_i$, the variables $R$ and $T_{\l}$ can be interpreted as having the same set theoretical meaning in the formulation, and we can deem $T_{\l}$ redundant. Therefore, in this case, $T_1,\ldots, T_k$ are all $P$-useless.}

With the notion of $P$-useless variables, we can show the following:

\begin{prop}\label{OMCuseless}
Let $\Gamma_1, \Gamma_2$ be two admissible lift-and-project operators, $P \subseteq [0,1]^n$, and suppose $\tilde{\Gamma}_1(P) \subseteq \mathbb{R}^{\S_1 \times \S'_1}$ and $\tilde{\Gamma}_2(P) \subseteq \mathbb{R}^{\S_2 \times \S'_2}$. Also, let $f_1,g_1$ and $f_2, g_2$ be the corresponding constraint functions of $\Gamma_1$ and $\Gamma_2$ respectively, and let $U$ be a set of $P$-useless variables in $\S'_2$.  Further suppose that
the following conditions hold:

\begin{itemize}
\item[(i)]
Every matrix in $\tilde{\Gamma}_1(P)$ satisfies (OMC).
\item[(ii)]
$\set{S \cap S' : S \in \S_1, S' \in \S'_1}$ refines $\set{S \cap S' : S \in \S_2 \setminus U, S' \in \S'_2 \setminus U}$, and $\S'_1$ refines \\$\S'_2 \setminus U$.
\item[(iii)]
Let $Y \in \tilde{\Gamma}_1(P)$, and $S \in \S'_2$. If $y \in \mathbb{R}^{\S_2 \times \set{S}}$ is consistent with $Y$, then $y \in f_2(S)$.
\item[(iv)]
If $Y_1 \in g_1(P)$ and $Y_2 \in \mathbb{R}^{\S_2 \times \S'_2}$ is consistent with $Y_1$, then $Y_2 \in g_2(P)$.
\end{itemize}
Then, $\Gamma_1(P) \subseteq \Gamma_2(P)$.
\end{prop}

Intuitively, the above conditions are needed so that given a point $x \in \Gamma_1(P)$ and its certificate matrix $Y \in \tilde{\Gamma}_1(P)$, we know enough structure about the entries and set theoretic meanings of $Y$ to construct a matrix $Y'$ in $\mathbb{R}^{(\S_2 \sm U) \times (\S'_2 \sm U)}$ that is consistent with $Y$. Then using the fact that the variables in $U$ are $P$-useless, we can extend $Y'$ to a matrix $Y''$ in $\mathbb{R}^{\S_2 \times \S'_2}$ that certifies $x$'s membership in $\Gamma_2(P)$. Also, for $y \in \mathbb{R}^{\S_2 \times \set{S}}$, we are referring to a vector with $|\S_2|$ entries that are indexed by elements of $\set{T \cap S: T \in \S_2}$. Since we will be talking about whether $y$ is consistent with another vector or matrix, we will need to specify not only the entries of $y$, but also these entries' corresponding sets.

Now we are ready to prove \autoref{OMCuseless}.

\begin{proof}[Proof of \autoref{OMCuseless}]
Suppose $x \in \Gamma_1(P)$. Let $Y \in \mathbb{R}^{\S_1 \times \S'_1}$ be a matrix in $\tilde{\Gamma}_1(P)$ such that $\het{x}(Ye_{\F}) = \het{x}$.
First, we construct an intermediate matrix $Y' \in \mathbb{R}^{(\S_2 \setminus U) \times (\S'_2 \setminus U)}$. For each $\a \in \S_2 \setminus U$ and $\b \in \S'_2 \setminus U$, we know (due to (ii)) that there exists a set of ordered pairs
\[
I_{\a,\b} \subseteq \set{(S, S') : S \in \S_1, S' \in \S'_1}
\]
such that the collection $\set{ S \cap S' : (S,S') \in I_{\a,\b}}$ partitions $\a \cap \b$. Next, we construct $Y'$ such that
\[
Y'[\a, \b] := \sum_{(S,S') \in I_{\a,\b}} Y[S,S'].
\]
Note that by (OMC), the entry $Y'[\a,\b]$ is invariant under the choice of $I_{\a,\b}$. Also, since $\set{ (\F,\F)}$ is a valid candidate for $I_{\F,\F}$, we see that $Y'[\F, \F] = Y[\F,\F] = 1$, and $\het{x}(Y'e_{\F}) = \het{x}(Ye_{\F}) = \het{x}$.

Next, we construct $Y'' \in \tilde{\Gamma}_2(P)$ from $Y'$. For each $\a \in U$ for which $P \cap \tn{conv}(\a) \neq \es$, we define a set $h(\a) \in \S_2' \setminus U$ such that $\tn{conv}(\a) \cap P = \tn{conv}(h(\a)) \cap P$. This can be done as follows: by the definition of $\a$ being $P$-useless, there must be a collection $\T = \set{T_1, \ldots, T_k, \a} \subseteq U$ where $\conv(T_i) \cap P = \es$ for all $i \in [k]$, and a set $R := \left( \bigcup_{i=1}^k T_i  \right) \cup \a \in \S'_1$ that satisfies $R \not \in \T$ and $\conv(\a) \cap P = \conv(R) \cap P$. If $R \not\in U$, then we can let $h(\a) = R$. Otherwise, since $R$ is itself $P$-useless, we can repeat the argument and find a yet larger set $R'$ where $\conv(R') \cap P = \conv(R) \cap P$. Since $U$ is finite, we can eventually find a set $h(\a) \in \S_2' \setminus U$ that has the desired property. Note that $h(\a)$ may not be unique, but any eligible choice would do.

Next, we define $V^1 \in \mathbb{R}^{(\S_2 \setminus U) \times \S_2}$ as follows:
\[
V^1(e_{\a}) := \left\{
\begin{array}{ll}
e_{\a} & \tn{if $\a \in \S_2 \setminus U$;}\\
e_{h(\a)} & \tn{if $\a \in U$ and $\tn{conv}(\a) \cap P \neq \es$;}\\
0 & \tn{otherwise.}
\end{array}
\right.
\]
Similarly, we define $V^2 \in \mathbb{R}^{(\S'_2 \setminus U) \times \S'_2}$ as follows:
\[
V^2(e_{\a}) := \left\{
\begin{array}{ll}
e_{\a} & \tn{if $\a \in \S'_2 \setminus U$;}\\
e_{h(\a)} & \tn{if $\a \in U$ and $\tn{conv}(\a) \cap P \neq \es$;}\\
0 & \tn{otherwise.}
\end{array}
\right.
\]
We show that $Y'' := V^1 Y' (V^2)^{\top} \in \tilde{\Gamma}_2(P)$. Since our map from $Y$ to $Y''$ preserves (OMC), $Y''$ is consistent with $Y$, and thus by (iv) it satisfies all constraints in $g_2$. Also, by (iii) it satisfies all column constraints in $f_2$ as well. Thus, $Y'' \in \tilde{\Gamma}_2(P)$. Since $\het{x}(Y'' e_{\F}) = \het{x}$, we are finished.
\end{proof}

We note that, in some cases, we can relate the performance of two lift-and-project operators by assuming a condition slightly weaker than (OMC). Given a matrix $Y \in \mathbb{R}^{\S \times \S'}$, where $\S, \S' \subseteq \A$, we say that it is \emph{row and column measure consistent} (RCMC) if every column and row of $Y$ satisfies (OMC). As is apparent in its definition, (RCMC) is less restrictive than (OMC). For example, consider
\[
Y := \begin{pmatrix}
Y[\F, \F] & Y[\F, 1|_1] & Y[\F, 2|_1] & Y[\F, 1|_0] & Y[\F, 2|_0]  \\
Y[1|_1, \F]& Y[1|_1, 1|_1] & Y[ 1|_1, 2|_1] & Y[ 1|_1, 1|_0] & Y[ 1|_1, 2|_0]  \\
Y[ 2|_1, \F] & Y[2|_1, 1|_1] & Y[2|_1, 2|_1] & Y[2|_1, 1|_0] & Y[2|_1, 2|_0]
\end{pmatrix} =
\begin{pmatrix}
1 & 0.7 & 0.4 & 0.3 & 0.6\\
0.7 & 0.7 & 0.2 & 0 & 0.5 \\
0.4 & 0.3 & 0.4 & 0.1 & 0
\end{pmatrix}.
\]
Then $Y$ satisfies (RCMC), but not (OMC) since $Y[1|_1, 2|_1] \neq Y[2|_1, 1|_1 ]$. It is not hard to check that all matrices in the lifted space of all named lift-and-project operators mentioned in this paper satisfy (RCMC). Next, we prove a result that is the (RCMC) counterpart of \autoref{OMCuseless}:

\begin{prop}\label{RCMCuseless}
Let $\Gamma_1, \Gamma_2$ be two admissible lift-and-project operators, $P \subseteq [0,1]^n$, and suppose $\tilde{\Gamma}_1(P) \subseteq \mathbb{R}^{\S_1 \times \S'_1}$ and $\tilde{\Gamma}_2(P) \subseteq \mathbb{R}^{\S_2 \times \S'_2}$. Also, let $f_1,g_1$ and $f_2, g_2$ be the corresponding constraint functions of $\Gamma_1$ and $\Gamma_2$ respectively, and let $U$ be a set of $P$-useless variables in $\S'_2$. Further suppose
that all of the following conditions hold:

\begin{itemize}
\item[(i)]
Every matrix in $\tilde{\Gamma}_1(P)$ satisfies (RCMC).
\item[(ii)]
$\S_1$ refines $\S_2 \setminus U$, and $\S'_1$ refines $\S'_2 \setminus U$.
\item[(iii)]
Let $S \in \S'_2$. If $x \in \mathbb{R}^{\S_1 \times \set{S}}$ is contained in $f_1(S)$ and
$y \in \mathbb{R}^{\S_2 \times \set{S}}$ is consistent with $x$, then $y \in f_2(S)$.
\item[(iv)]
If $Y_1 \in g_1(P)$ and $Y_2 \in \mathbb{R}^{\S_2 \times \S'_2}$ is consistent with $Y_1$, then $Y_2 \in g_2(P)$.
\end{itemize}
Then, $\Gamma_1(P) \subseteq \Gamma_2(P)$.
\end{prop}

\begin{proof}
The result can be shown by following the same outline as in the proof of \autoref{OMCuseless}. Suppose $x \in \Gamma_1(P)$ and $Y \in \mathbb{R}^{\S_1 \times \S'_1}$ is a certificate matrix for $x$. For each $\a \in \S_2 \setminus U$, define $I_{\a}$ to be a collection of sets in $\S_1$ that partitions $\a$. Since $\S_1$ refines $\S_2 \setminus U$, such a collection must exist. Likewise, for all $\a \in \S'_2 \sm U$, we define $I'_{\a}$ to be a collection of sets in $\S'_1$ that partitions $\a$.

Next, we define $Y' \in \mathbb{R}^{(\S_2 \sm U) \times (\S'_2 \sm U)}$ such that
\[
Y'[\a, \b] := \sum_{S \in I_{\a}, S' \in I'_{\b}} Y[S,S'].
\]
Since $Y$ satisfies (RCMC), $Y'[\a,\b]$ is invariant under the choices of $I_{\a}$ and $I'_{\b}$. From here on, we can define $V_1, V_2$ and $Y'' \in \mathbb{R}^{\S_2 \times \S'_2}$ as in the proof of \autoref{OMCuseless}, and apply the same reasoning therein to show that it is in $\tilde{\Gamma}_2(P)$. Now since $\het{x}(Y'' e_{\F}) = \het{x}(Y e_{\F}) = \het{x}$, we conclude that $x \in \Gamma_2(P)$.
\end{proof}

\subsection{Implications and Applications}

Next, we look into several implications of~\autoref{OMCuseless} and~\autoref{RCMCuseless}. First, it is apparent that given two operators $\Gamma_1, \Gamma_2$ and a set $P$ such that $\Gamma_1(P) \subseteq \Gamma_2(P)$, the integrality gap of $\Gamma_1(P)$ is no more than that of $\Gamma_2(P)$ with respect to any chosen direction. We will formally define integrality gaps and discuss these results in more depth in Section 5.

Next, we relate the performance of $\BZ'$ and $\SA'$ under some suitable conditions. First, we define a tier $S \in \T_k$ to be $P$-useless if all variables associated with $S$ are $P$-useless. Then we have the following:

\begin{tm}\label{SABZ'}
Suppose there exists $\l \in [n]$ such that all tiers $S$ generated by $\BZ'^k$ of size greater than $\l$ are $P$-useless. Then
\[
\BZ'^k(P) \supseteq \SA'^{2\l}(\O_k(P)).
\]
\end{tm}

\begin{proof}
Let $\Gamma_1 = \SA'^{2\l}(\O_k(\cdot))$ and $\Gamma_2 = \BZ'^k(\cdot)$. We prove our assertion by checking all conditions listed in \autoref{OMCuseless}.

First of all, for every set $P \subseteq [0,1]^n$, all matrices in the lifted space of $\SA'^{2\l}(\O_k(P))$ satisfy (OMC). Next, since $\S_1 = \A_1^+$ and $\S_1' = \A_{2\l}$, we see that $\set{S \cap S' : S \in \S_1, S' \in \S_1'}$ refines $\A_{2\l}$. On the other hand, since every tier of size greater than $\l$ is $P$-useless, we see that $\A_{\l}$ refines both $\S_2 \sm U$ and $\S_2' \sm U$. Thus, $
\A_{2\l} =\set{ S \cap S' : S, S'\in \A_{\l} }$ refines $\set{ S \cap S' : S \in \S_2 \sm U, S' \in \S'_2 \sm U}$. Also, it is apparent that $\S'_1 = \A_{2 \l}$ refines $\S'_2 \sm U$, so (ii) holds.

For (iii), we let $f_1(S) = K(\O_k(P) \cap \tn{conv}(S)),~\forall S \in \A$, and
\[
f_2(S) := \set{y \in \mathbb{R}^{\S'_2} : \het{x}(y) \in K(\O_k(P) \cap \tn{conv}(S)), \tn{$y$ satisfies $(\BZ'2)$}}.
\]
Note that all conditions in $(\BZ' 2)$ are relaxations of constraints in (P5) and (P6), and thus are implied by (OMC). Let $Y \in \tilde{\SA}'^{2\l}(\O_k(P))$, and $Y''$ be the matrix obtained from the construction in the proof of \autoref{OMCuseless}. Since $Y$ satisfies (OMC), so does $Y''$ (as it is consistent with $Y$). Also, since all conditions in ($\BZ' 2$) are implied by (OMC), the columns of $Y''$ must satisfy ($\BZ' 2$).

To check (iv), we see that $g_2(P)$ would be the set of matrices in the lifted space that satisfy $(\BZ' 3)$ and $(\BZ' 4)$. It is easy to see that $(\BZ' 4)$ is implied by (OMC). For $(\BZ' 3)$, suppose $S \in \S_2, S' \in \S'_2$, and $\conv(S) \cap \conv(S') \cap \O_k(P) = \es$.
If $Y''[S, S'] \neq 0$, then we know that $P \cap \tn{conv}(S) \neq \es$ and $P \cap \tn{conv}(S') \neq \es$, by the construction of $Y''$. Thus, define $\a := S$ if $S \not \in U$, and $\a := h(S)$ if $S \in U$. Likewise, define $\b := S'$ if $S' \not\in U$, and $\b := h(S')$ if $S' \in U$. In all cases, we have now obtained $\a \in \S_2 \sm U, \b \in \S'_2 \sm U$ such that $Y''[\a,\b] = Y''[S,S']$.

Since
\[
Y''[\a,\b] = Y'[\a,\b] = \sum_{(T,T') \in I_{\a,\b}} Y[T,T'],
\]
we obtain $T \in \A_1^+,T' \in \A_k$ such that $Y[T,T'] \neq 0$. Then by ($\SA' 4$), $\conv(T) \cap \conv(T') \cap P \neq \es$. This implies that $\conv(S) \cap \conv(S') \cap P \neq \es$, and so ($\BZ' 3$) holds.
\end{proof}

We remark that, with a little more care and using the same observation as in the proof of \autoref{SASA'}, one can slightly sharpen \autoref{SABZ'} and show that $\SA^{2\l}(\O_k(P)) \subseteq \BZ'^k(P)$ under these assumptions.

Next, we look into the lift-and-project ranks of a number of relaxations that arise from
combinatorial optimization problems. For any lift-and-project operator $\Gamma$ and polytope $P$, we define the \emph{$\Gamma$-rank of $P$} to be the smallest integer $k$ such that $\Gamma^k(P) = P_I$. The notion of rank gives us a measure of how close $P$ is to $P_I$ with respect to $\Gamma$. Moreover, it is useful when comparing the performance of different operators applied to the same $P$.

Given a simple, undirected graph $G = (V,E)$, we define
\[
MT(G) := \set{x \in [0,1]^E : \sum_{j :\set{i,j} \in E} x_{ij} \leq 1,~\forall i \in V}.
\]
Then $MT(G)_I$ is the \emph{matching polytope} of $G$, and is exactly the convex hull of incidence vectors of matchings of $G$.

While there exist efficient algorithms that solve the matching problem (e.g. Edmonds' seminal blossom algorithm~\cite{Edmonds65a}), many lift-and-project operators have been shown to require exponential time to compute the matching polytope starting with $MT(G)$. In particular, $MT(K_{2n+1})$ is known to have $\LS_+$-rank $n$~\cite{StephenT99a} and  $\BCC$-rank $n^2$~\cite{AguileraBN04a}. More recently, Mathieu and Sinclair~\cite{MathieuS09a} showed that the $\SA$-rank of $MT(K_{2n+1})$ is $2n-1$. Using their result and \autoref{SABZ'}, we can show that this polytope is also a bad instance for $\BZ'$.

\begin{tm}\label{BZ'MTG}
The $\BZ'$-rank of $MT\left(K_{2n+1}\right)$ is at least $\left\lceil \sqrt{2n} - \frac{3}{2} \right\rceil$.
\end{tm}

\begin{proof}
Let $G = K_{2n+1}$ and $P = MT(G)$. We first identify the tiers generated by $\BZ'^k$ that are $P$-useless. Observe that a set $O \subseteq E$ is a $k$-small obstruction generated by $\BZ'^k$ if there is a vertex that is incident with all edges in $O$, and that $2 \leq |O| \leq k+1$ or $|O| \geq 2n-k$. Now suppose $W \in \W_k$ is a wall, and let $\set{e_1, e_2, \ldots, e_p}$ be a maximum matching contained in $W$. Notice that for $e_1 = \set{u_1,v_1}$ to be in $W$, it has to be contained in at least two obstructions and each of these obstructions has to originate from the $u_1$- or $v_1$-constraint in the formulation of $MT(G)$. Now suppose $e_2 = \set{ u_2, v_2}$. By the same logic, we deduce that the obstructions that allow $e_2$ to be in $W$ have to be different from those that enabled $e_1$ to be in $W$. Since each wall is generated by at most $k+1$ obstructions, we see that $p \leq \frac{k+1}{2}$. Therefore, for every tier $S \in \T_k$ (which has to be contained in the union of $k$ walls), the maximum matching contained in $S$ has at most $\frac{k(k+1)}{2}$ edges.

Hence, if a tier $S$ has size greater than $\frac{k(k+1)}{2} + k$, then $S \sm T$ is not a matching for any set $T \subseteq S$ of size up to $k$, which implies $\conv{(S \sm T)|_1} \cap P = \es$, and so $\conv{(S \sm T)|_1 \cap T|_0} \cap P = \es$.  Thus, the only variables $\a$ associated with $S$ such that $\conv(\a) \cap P \neq \es$ take the form $\a = (S \sm (T \cup U))|_1 \cap T|_0 \cap U|_{< |U| -(k-|T|)}$ for some disjoint sets $U,T$ where $|T| < k$ and $|U| + |T| > k$. Next, observe that $(S \sm (T \cup U))|_1 \cap T|_0$ is partitioned by $\a$ and the sets
\begin{equation}\label{BZ'MTGeq1}
(S \sm (T \cup U))|_1 \cap T|_0 \cap U'|_1 \cap (U \sm U')|_0
\end{equation}
where $U' \geq |U| -(k-|T|)$. Also, since $S$ is a tier generated by $\BZ'^k$, so is its subset $S \sm U$, and we see that the variable $(S \sm (T \cup U))|_1 \cap T|_0$ is present. Thus, every set in~\eqref{BZ'MTGeq1}, together with $\a$, are $P$-useless. Since this argument applies for all $\a$'s in the above form, we see that all variables associated with $S$ are $P$-useless.

Since it was shown in~\cite{MathieuS09a} that $P$ has $\SA$-rank $2n-1$, it follows from ~\autoref{SASA'} that the $\SA'$-rank of $P$ is at least $2n-2$. Thus, by~\autoref{SABZ'}, for $\BZ'^k(P)$ to be equal to $P_I$, we need $2\left( \frac{k(k+1)}{2}  + k \right) \geq 2n -2$.  Therefore, $k \geq \sqrt{2n} -\frac{3}{2}$.
\end{proof}

The best upper bound we know for the $\BZ'$-rank of $MT(K_{2n+1})$ is $2n-1$ (due to Mathieu and Sinclair's result, and the fact that $\BZ'^k$ dominates $\SA^k$). We shall see in the next section that strengthening $\BZ'$ by an additional positive semidefiniteness constraint decreases the current best upper bound to roughly $\sqrt{2n}$.

We next look at the stable set problem of graphs. Given a graph $G = (V,E)$, its \emph{fractional stable set polytope} is defined to be
\[
FRAC(G) := \set{ x \in [0,1]^V : x_i + x_j \leq 1,~\forall \set{i,j} \in E}.
\]
Then the \emph{stable set polytope} $STAB(G) := FRAC(G)_I$ is precisely the convex hull of incidence vectors of stable sets of $G$. Since there is a bijection between the set of matchings in $G$ and the set of stable sets in its line graph $L(G)$, the next result follows readily from \autoref{BZ'MTG}.

\begin{cor}\label{BZ'LG}
Let $G$ be the line graph of $K_{2n+1}$. Then the $\BZ'$-rank of $FRAC(G)$ is at least $\left\lceil \sqrt{2n} - \frac{3}{2} \right\rceil$.
\end{cor}

\begin{proof}
First, it is not hard to see that $MT(H) \subseteq FRAC(L(H))$, for every graph $H$. Also, since the collection of $k$-small obstructions of $FRAC(G)$ is exactly the set of edges of $G$ for all $k \geq 1$, we see that $FRAC(G) = \O_k(FRAC(G))$. Therefore,
\[
\O_k(MT(K_{2n+1})) \subseteq MT(K_{2n+1}) \subseteq FRAC(G) = \O_k(FRAC(G)).
\]
This, together with the fact that every $k$-small obstruction of $FRAC(G)$ is also a $k$-small obstruction of $MT(K_{2n+1})$, implies that $\BZ'^k(MT(K_{2n+1})) \subseteq \BZ'^k(FRAC(G))$. Thus, the $\BZ'$-rank of $FRAC(G)$ is at least that of $MT(K_{2n+1})$, and our claim follows.
\end{proof}

Thus, we obtain from \autoref{BZ'LG}, a family of graphs on $n$ vertices whose fractional stable set polytope has $\BZ'$-rank $\Omega(n^{1/4})$.

We next turn to the complete graph $G := K_n$. It is well known that $FRAC(G)$ has rank $\Theta(n)$ with respect to $\SA$ (and as a result, all weaker operators such as $\LS$ and $\LS_0$). We show that this is also true for $\BZ'$. 

\begin{tm}\label{BZrankofFRACG}
The $\BZ'$-rank of $FRAC(K_n)$ is between $\left\lceil \frac{n}{2} \right\rceil -2$ and $\left\lceil \frac{n+1}{2}  \right\rceil $, for all $n \geq 3$. The same bounds apply for the $\BZ$-rank.
\end{tm}

The proof of \autoref{BZrankofFRACG} will be provided in the Appendix. Thus, we see that, like all other popular polyhedral lift-and-project operators, $\BZ'$ (which is already stronger than $\BZ$) performs poorly on the fractional stable set polytope of complete graphs.

\section{Tools for analyzing Lift-and-Project Operators\\with Positive Semidefiniteness}

Up to this point, we have looked exclusively at lift-and-project operators that produce polyhedral relaxations, where the main tool operators use to gain strength is to lift a given relaxation to a higher dimensional space. In this section, we turn our focus to operators that do not produce polyhedral relaxations. In particular, we will introduce several lift-and-project operators that utilize positive semidefiniteness, and look into the power and limitations of these additional constraints.

\subsection{Lift-and-Project Operators with Positive Semidefiniteness}

Perhaps the most elementary operator of this type is the $\LS_+$ operator defined in~\cite{LovaszS91a}. Recall that one way to see why $P_I \subseteq \LS(P)$ in general is to observe that for any integral point $x \in P$, $\het{x}\het{x}^{\top}$ is a matrix that certifies $x$'s membership in $\LS(P)$. Since $\het{x}\het{x}^{\top}$ is positive semidefinite for all $x$, if we let $\mathbb{S}_+^{n} \subset \mathbb{S}^n$ denote the set of symmetric, positive semidefinite $n$-by-$n$ matrices, then it is easy to see that
\[
\LS_+(P) := \set{x \in \mathbb{R}^n : \exists Y \in \mathbb{S}_+^{n+1}, Ye_i,  Y(e_0 - e_i)\in K(P),~\forall i \in [n], Ye_0 = \tn{diag}(Y) = \het{x}}
\]
contains $P_I$ as well. Also, by definition, $\LS_+(P) \subseteq \LS(P)$ for all $P \subseteq [0,1]^n$, and thus $\LS_+$ potentially obtains a tighter relaxation than $\LS(P)$ in general.

Likewise, we can also define positive semidefinite variants of $\SA$. Given any positive integer $k$, we define the operators $\SA_+^k$ and $\SA_+'^{k}$ as follows:

\begin{enumerate}
\item
Let $\tilde{\SA}_+^k(P)$ denote the set of matrices $Y \in \mathbb{S}_+^{\A_{k}}$ that satisfy all of the following conditions:
\begin{itemize}
\item[($\SA_+ 1$)]
$Y[\F, \F] = 1$.
\item[($\SA_+ 2$)]
For every $\a \in \A_k$:
\begin{itemize}
\item[(i)]
$\het{x}(Ye_{\a}) \in K(P)$;
\item[(ii)]
$Ye_{\a} \geq 0$.
\end{itemize}
\item[($\SA_+ 3$)]
For every $S|_1 \cap T|_0 \in \A_{k-1}$,
\[
Ye_{S|_1 \cap T|_0 \cap j|_1} + Ye_{S|_1 \cap T|_0 \cap j|_0} = Ye_{S|_1 \cap T|_0}, \quad \forall j \in [n] \sm (S \cup T).
\]
\item[($\SA_+4$)]
For all $\a,\b \in \A_k$ such that $\a \cap \b = \es, Y[\a,\b] = 0$.
\item[($\SA_+5$)]
For all $\a_1,\a_2, \b_1, \b_2 \in \A_k$ such that $\a_1 \cap \b_1 = \a_2 \cap \b_2, Y[\a_1, \b_1] = Y[\a_2, \b_2]$.
\end{itemize}
\item
Let $\tilde{\SA}_+'^{k}(P)$ be the set of matrices $\tilde{\SA}_+^k(P)$ that also satisfy:
\begin{itemize}
\item[($\SA_+' 4$)]
For all $\a,\b \in \A_k$ such that $\conv(\a) \cap \conv(\b) \cap P = \es, Y[\a,\b] = 0$.
\end{itemize}
\item
Define
\[
\SA_+^k(P) = \set{x \in \mathbb{R}^n: \exists Y \in \tilde{\SA}_+^k(P),  \het{x}(Ye_{\F})= \het{x}},
\]
and
\[
\SA_+'^{k}(P) := \set{x \in \mathbb{R}^n: \exists Y \in \tilde{\SA}_+'^{k}(P),  \het{x}(Ye_{\F})= \het{x}}.
\]
\end{enumerate}

The $\SA_+^k$ and $\SA_+'^{k}$ operators extend the lifted space of the $\SA^k$ operator to a set of square matrices, and impose an additional positive semidefiniteness constraint. What sets these two new operators apart is that $\SA_+'^{k}$ utilizes a $(\BZ' 3)$-like constraint to potentially obtain additional strength over $\SA_+^k$. While we have seen in their polyhedral counterparts $\SA'$ and $\SA$ that adding this additional constraint could decrease the rank of a polytope by at most one, we shall provide an example later in this section in which the $\SA'_+$-rank of a polytope is lower than the $\SA_+$-rank by $\Theta(n)$.

Note that in ($\SA_+2$) we have imposed that all certificate matrices in $\tilde{\SA}_+^k(P)$ (which contains $\tilde{\SA}_+'^k(P)$) have nonnegative entries, which obviously holds for matrices lifted from integral points. In contrast with ($\SA 2$), the nonnegativity condition was not explicitly stated there as it is implied by the fact that $P \subseteq [0,1]^n$.

It is well known that $\SA^k(P) \subseteq \LS(\SA^{k-1}(P))$ for all polytopes $P \subseteq [0,1]^n$ and for all $k \geq 1$ (see, for instance, Theorem 12 in~\cite{Laurent03a} for a proof). It then follows that $\SA^k$ dominates $\LS^k$ for all $k \geq 1$. Using very similar ideas, we prove an analogous result for the semidefinite counterparts of these operators:

\begin{prop}\label{LS_+SA_+}
For every polytope $P \subseteq [0,1]^n$ and every integer $k \geq 1$,
\[
\SA_+^k(P) \subseteq \LS_+(\SA_+^{k-1}(P)).
\]
\end{prop}

\begin{proof}
Suppose $Y \in \tilde{\SA}_+^k(P)$ and $\het{x}(Ye_{\F}) = \het{x}$. Let $Y'$ be the $(n+1)$-by-$(n+1)$ symmetric minor of $Y$, with rows and columns indexed by elements in $\A^+_1$. To adapt to the notation for $\LS_+$, we index the rows and columns of $Y'$ by $0,1,\ldots, n$ (instead of $\F, 1|_1, \ldots, n|_1$). It is obvious that $Y' \in \mathbb{S}_+^{n+1}$, and $Y'e_0 = \tn{diag}(Y') = \het{x}$. Thus, it suffices to show that $Y'e_i,  Y'(e_0 - e_i)\in K(\SA_+^{k-1}(P)),~\forall i \in [n]$.

We first show that $Y'e_i \in K(\SA_+^{k-1}(P))$. If $(Y'e_i)_0 = 0$, then $Y'e_i$ is the zero vector and the claim is obviously true. Next, suppose $(Y'e_i)_0 >0$. Define the matrix $Y'' \in \mathbb{S}^{\A_{k-1}}$, such that
\[
Y''[\a,\b] = \frac{1}{(Y'e_i)_0}Y[\a \cap i|_1, \b \cap i|_1], \quad \forall \a,\b \in \A_{k-1}.
\]
Notice that $Y''$ is a positive scalar multiple of a symmetric minor of $Y$, and thus is positive semidefinite. Moreover, it satisfies ($\SA_+ 1$) by construction, and inherits the properties ($\SA_+ 2$) to  ($\SA_+ 5$) from $Y$. Thus, $Y'' \in \tilde{\SA}_+^{k-1}(P)$ and $\het{x}(Y'' e_{\F}) = \frac{1}{(Y'e_i)_0}Y'e_i \in K(\SA_+^{k-1}(P))$. The argument for $Y'(e_0 - e_i)$ is analogous.
\end{proof}

It follows immediately from \autoref{LS_+SA_+} that $\SA_+^k(P) \subseteq \LS_+^k(P)$, and thus $\SA_+^k$ dominates $\LS_+^k$. The $\SA_+$ and $\SA'_+$ operators will be useful in simplifying our analysis and improving our understanding of the Bienstock--Zuckerberg operator enhanced with positive semidefiniteness, which is defined as
\[
\BZ_{+}'^{k}(P) := \set{x \in \mathbb{R}^n: \exists Y \in \tilde{\BZ}_+'^{k}(P),  \het{x}(Ye_{\F})= \het{x}},
\]
where $\tilde{\BZ}_+'^{k}(P) := \tilde{\BZ}'^{k}(P) \cap \mathbb{S}_+^{\A'}$.

\subsection{Unhelpful variables in PSD relaxations}

We see that in \autoref{RCMCuseless}, in the special case of comparing two lift-and-project operators whose lifted spaces are both square matrices (i.e. $\S_1 = \S'_1$ and $\S_2 = \S'_2$), the construction of $Y'$ and $Y''$ preserves positive semidefiniteness of $Y$. Thus, this framework can be applied even when $g_1$ and $g_2$ enforce positive semidefiniteness constraints in their respective lifted spaces. The following is an illustration of such an application:

\begin{tm}\label{SA'+BZ'+}
Suppose there exists $\l \in [n]$ such that all tiers $S$ generated by $\BZ_+'^{k}$ of size greater than $\l$ are $P$-useless. Then
\[
\BZ_{+}'^k(P) \supseteq \SA_+'^{\l}(\O_k(P)).
\]
\end{tm}

\begin{proof}
We prove our claim by verifying the conditions in \autoref{RCMCuseless}. First, every matrix in the lifted space of $\SA_+'^{\l}$ satisfies (OMC), which implies (RCMC). Next, since $\S_1 = \S'_1 = \A_{\l}$ and every tier of $\BZ_+'^{k}$ that is not useless has size at most $\l$, we see that (ii) holds as well.

For (iii), note that we can let
\[
f_1(S) = \set{y \in \mathbb{R}^{\S'_1} : \het{x}(y) \in K(P \cap \tn{conv}(S)), \tn{$y$ satisfies (OMC)}},
\]
and
\[
f_2(S) = \set{y \in \mathbb{R}^{\S'_2} : \het{x}(y) \in K(P \cap \tn{conv}(S)), \tn{$y$ satisfies $(\BZ'2)$}}.
\]
As mentioned before, all conditions in $(\BZ' 2)$ are implied by (OMC) constraints and the fact that $\A_{\l}$ refines $\S_2$. Thus, (iii) is satisfied.

For (iv), we see that $g_2(P)$ would be the set of matrices in $\mathbb{S}_+^{\S_2}$ that satisfy $(\BZ' 3)$ and $(\BZ' 4)$. It is easy to see that $(\BZ' 4)$ is implied by (OMC). Also, $(\BZ' 3)$ is implied by $(\SA_+' 4)$. Thus, we are finished.
\end{proof}

\subsection{Utilizing $\l$-establishing variables}

Somewhat complementary to the notion of useless variables, here we look into instances where the presence of a certain set of variables in the lifted space provides a guarantee on the overall performance of the operator.  Given $j \in \set{0,1,\ldots, n}$, let $[n]_j$ denote the collection of subsets of $[n]$ of size $j$. Suppose $Y \in \mathbb{S}^{\A'}$ for some $\A' \subseteq \A$, and there exists a positive integer $\l$ where all of the following conditions hold:

\begin{itemize}
\item[($\l1$)]
$Y[\F, \F] = 1$.
\item[($\l2$)]
$Y \succeq 0$.
\item[($\l3$)]
$\A^+_{\l} \subseteq \A'$.
\item[($\l4$)]
For all $\a,\b, \a', \b' \in \A^+_{\l}$ such that $\a \cap \b = \a' \cap \b', Y[\a, \b] = Y[\a',\b']$.
\item[($\l5$)]
For all $\a,\b \in \A^+_{\l}, Y[\F,\b] \geq Y[\a,\b]$.
\end{itemize}

Then we say that such a matrix $Y$ is \emph{$\l$-established}. Notice that all matrices in $\tilde{\SA}_+^{k}(P)$ (which contains $\tilde{\SA}_+'^{k}(P)$) are $\l$-established, for all $P \subseteq [0,1]^n$. A matrix in $\tilde{\BZ}_{+}'^{k}(P)$ is $\l$-established if all subsets of size up to $\l$ are generated as tiers.
Given such a matrix, we may define a vector $y$ whose entries are indexed by the sets $\bigcup_{i=0}^{2\l} [n]_i$  such that $y_{S} = Y[S'|_1, S''|_1]$, where $S', S''$ are subsets of $[n]$ of size at most $\l$ such that $S' \cup S'' = S$. Note that such choices of $S', S''$ must exist by ($\l 3$), and by ($\l 4$) the value of $y_{S}$ is invariant under the choices of $S'$ and $S''$.

Finally, we define $Z \in \mathbb{R}^{2\l +1}$ such that
\[
Z_i := \sum_{S \subseteq [n]_i} y_{S}, \quad \forall i \in \set{0,1,\ldots , 2\l}.
\]
Note that $Z_0$ is always equal to $1$ (by ($\l1$)), and $Z_1 = \sum_{i=1}^n Y[i|_1, \F]$. Also, observe that the entries of $Z$ are related to each other. For example, if $\het{x}(Ye_{\F})$ is an integral $0$,$1$ vector, then by ($\l5$) we know that $y_S \leq 1$ for all $S$, and $y_{S} >0$ only if $y_{\set{i}}=1,~\forall i \in S$. Thus, we can infer that
\[
Z_j = \sum_{S \in [n]_j} y_S \leq \binom{Z_1}{j}, \quad \forall j \in [2\l].
\]
We next show that the positive semidefiniteness of $Y$ also forces the $Z_i$'s to relate to each other, somewhat similarly to the above. The following result would be more intuitive by noting that $\binom{p}{i+1}/\binom{p}{i} = \frac{p-i}{i+1}$.

\begin{prop}\label{Ziup}
Suppose $Y \in \mathbb{S}_+^{\A'}$ is $\l$-established, and $y, Z$ are defined as above. If there exists an integer $p \geq \l$ such that
\[
Z_{i+1} \leq \left( \frac{p-i}{i+1}\right) Z_i, \quad \forall i \in \set{\l, \l+1, \ldots, 2\l-1},
\]
then $Z_i \leq \binom{p}{i},~\forall i \in [2\l]$. In particular, $Z_1 \leq p$.
\end{prop}

\begin{proof}
We first show that $Z_{\l} \leq \binom{p}{\l}$. Given $i \in [\l]$, define the vector $v(i) \in \mathbb{R}^{\A'}$
such that
\[
v(i)_{\a} := \left\{
\begin{array}{ll}
\binom{p}{i} & \tn{if $\a = \F$;}\\
-1 & \tn{if $\a = S|_1$ where $S \in [n]_i$;}\\
0 & \tn{otherwise.}
\end{array}
\right.
\]
By the positive semidefiniteness of $Y$, we obtain
\begin{equation}\label{Ziup1}
0 \leq v(\l)^{\top}Yv(\l) = \binom{p}{\l}^2 - 2 \binom{p}{\l} Z_{\l} + \sum_{S, S' \in [n]_{\l}} Y[S|_1, S'|_1].
\end{equation}
Notice that for every $T \in [n]_{\l+j}$, the number of sets $T', T'' \in [n]_{\l}$ such that $T' \cup T'' = T$ is $\binom{\l}{j}\binom{\l+j}{\l}$. Hence, this is the number of times the term $y_{T}$ appears in $\sum_{S, S' \in [n]_{\l}} Y[S|_1, S'|_1]$. We also know by assumption
\begin{equation}\label{Ziup2}
Z_{\l +j} \leq \left( \frac{p-j-\l+1}{j+\l} \right) \left( \frac{p-j-\l+2}{j+\l -1}\right) \cdots \left( \frac{p-\l}{\l+1} \right) Z_{\l}= \frac{\binom{p-\l}{j} }{\binom{\l+j}{\l}} Z_{\l}
\end{equation}
for all $j \in [\l]$. Note that if $p < 2\l$, then by assumption we have $Z_{p+1} \leq \frac{p-p}{p+1} Z_{p} = 0$. As a result, $Z_{2\l} = Z_{2\l -1} =  \cdots = Z_{p+2} = Z_{p+1} =0$. In such cases,~\eqref{Ziup2} still holds as $\binom{p-\l}{j}$ would evaluate to zero. Then we have,
\begin{eqnarray*}
\sum_{S,S' \in [n]_{\l}} Y[S|_1, S'|_1]
&=& \sum_{j=0}^{\l} \sum_{S \in [n]_{\l+j}}  \binom{\l+j}{\l} \binom{\l}{j} y_{S}\\
&=&  \sum_{j=0}^{\l} \binom{\l+j}{\l}  \binom{\l}{j}  Z_{\l + j}\\
&\leq & \sum_{j=0}^{\l}  \binom{\l+j}{\l} \binom{\l}{j} \frac{\binom{p-\l}{j} }{\binom{\l+j}{\l}} Z_{\l}\\
&=& \binom{p}{\l}Z_{\l}.
\end{eqnarray*}
Therefore, we conclude from~\eqref{Ziup1} that $0 \leq \binom{p}{\l}^2 - \binom{p}{\l} Z_{\l}$, which implies that $Z_{\l} \leq \binom{p}{\l}$. Together with~\eqref{Ziup2}, this implies that $Z_{\l+j} \leq \binom{p}{\l +j},~\forall j \in \{0,1, \ldots,\l\}$.

It remains to show that $Z_i \leq \binom{p}{i},~\forall i \in [\l-1]$. To do that, it suffices to show that $Z_{i} \leq \binom{p}{i}$ can be deduced from assuming $Z_{i + j} \leq \binom{p}{i +j},~\forall j \in [i]$. Then applying the argument recursively would yield the result for all $i$. Observe that
\begin{eqnarray*}
\sum_{S, S' \in [n]_i} Y[S|_1 , S'|_1]
&=&  \sum_{j=0}^{i} \sum_{S \in [n]_{i+j}} \binom{i+j}{i} \binom{i}{j} y_{S}\\
&\leq & Z_{i} +  \sum_{j=1}^{i} \binom{i+j}{i} \binom{i}{j} \binom{p}{i+j}\\
&= & Z_{i} - \binom{p}{i} +  \sum_{j=0}^{i} \binom{i+j}{i} \binom{i}{j} \binom{p}{i+j}\\
&= & Z_{i} - \binom{p}{i} + \binom{p}{i} \left( \sum_{j=0}^{i} \binom{i}{j} \binom{p-i}{j} \right)\\
&=&  Z_{i} - \binom{p}{i} + \binom{p}{i}^2.
\end{eqnarray*}
Hence,
\[
0 \leq v(i)^{\top}Yv(i) \leq \binom{p}{i}^2 - 2 \binom{p}{i}Z_{i} + \left( Z_{i} -\binom{p}{i} + \binom{p}{i}^2\right) = \left( 2 \binom{p}{i} -1 \right)\left(\binom{p}{i} - Z_{i}\right),
\]
and we conclude that $Z_{i} \leq \binom{p}{i}$.
\end{proof}

%


An immediate but noteworthy implication of \autoref{Ziup} is the following:

\begin{cor}\label{Ziupcor}
Suppose $Y \in \mathbb{S}^{\A'}$ is $\l$-established, and $y, Z$ are defined as before. If $Z_i = 0,\\~\forall i > \l$, then $Z_1 \leq \l$.
\end{cor}

\begin{proof}
Since $Z_i = 0$ for all $i \in \set{\l+1, \ldots, 2\l}$, we can apply \autoref{Ziup} with $p=\l$ and deduce that $Z_i \leq \binom{\l}{i},~\forall i \in [2\l]$. In particular, $Z_1 \leq \l$.
\end{proof}

Note that~\autoref{Ziupcor} is somewhat similar in style to Theorem 13 in~\cite{KarlinMN11a}, which decomposes and reveals some structure of solutions in Lasserre relaxations using the fact that certain entries of the matrix in the lifted space are known to be zero.  These results were independently obtained.

We now employ the upper-bound proving techniques presented earlier and the notion of $\ell$-established matrices to prove the following result on the matching polytope of graphs.

\begin{tm}\label{SA+MTup}
The $\SA'_+$-rank of $MT(K_{2n+1})$ is at most $n - \left\lfloor \frac{\sqrt{2n+1} - 1}{2} \right\rfloor$.
\end{tm}

\begin{proof}
Let $G = K_{2n+1}$ and $P = MT(G)$. Let $Y \in \tilde{\SA}_+'^{k}(P)$. Since $Y$ is $k$-established, it suffices to show that $Z_{i+1} \leq \left( \frac{n-i}{i+1}\right) Z_i$ for all integer $i \in \set{k, k+1, \ldots, 2k-1}$ whenever $k \geq n - \left\lfloor \frac{\sqrt{2n+1} -1}{2} \right\rfloor$. Then it follows from \autoref{Ziup} that $Z_1 \leq n$, which implies $\sum_{i \in E(G)} x_i \leq n$ is valid for $\SA_+'^k(P)$.

First, by symmetry of the complete graph, we may assume that
\[
Y[S|_1, T|_1] = Y[S'|_1, T'|_1]
\]
whenever $S \cup T$ and $S' \cup T'$ are both matchings of $G$ of the same size. Thus, if we let $\M_{i}$  denote the set of all matchings of size $i$ in $G$, and $S \cup T$ is a matching of size $\l$ in $G$, we may assume that
\[
Y[S|_1, T|_1] = y_{S \cup T} = \frac{Z_{\l}}{|\M_{\l}|}.
\]
The last equality follows from our observation that by symmetry, we may assume that $y_{M}$ is identical for all $M \in \M_{\l}$, and the definition of $Z_{\l}$. Next, by the fact that the maximum cardinality matchings in $G$ have cardinality $n$ and the condition ($\SA_+' 4$),  $Z_i = 0,~\forall i > n$.
Thus, it suffices to verify the above claim for the case when $k \leq i \leq n-1$. Let $S$ be a matching of size $k$ that saturates the vertices $\set{2n-2k+2, \ldots, 2n+1}$, let $T$ be a matching of size $i-k$ that saturates vertices $\set{2n-2i+2, \ldots, 2n-2k+1}$, and let $E'$ be the set of edges in the subgraph of $G$ induced by the vertices $\set{1,2, \ldots, 2n-2i+1}$. Note that $E'$ contains exactly the edges that are not incident with vertices saturated by edges in $S$ or $T$. Also, for each $U \subseteq E'$, we define the vector $f_U \in \mathbb{R}^{|E'|+1}$ (indexed by $\set{0} \cup E'$) such that
\[
(f_U)_i := \left\{
\begin{array}{ll}
Y[(T \cup U)|_1 \cap (E' \sm U)|_0, S|_1]  & \tn{if $i=0$ or if $i \in U$;}\\
0 & \tn{otherwise.}
\end{array}
\right.
\]
Notice that $k \geq n - \frac{\sqrt{2n+1} - 1}{2}$ implies $k \geq \binom{2n+1-2k}{2} \geq |E'| + |T|$. Therefore, the above entries in $Y$ do exist, and the vectors $f_U$ are well-defined. Now notice that if $U \subseteq E'$ and $e \in E' \sm U$,
\begin{eqnarray*}
&& (f_{U \cup \set{e}})_0 + (f_{U})_0\\
 &=&  Y[((T \cup (U \cup \set{e}))|_1 \cap (E' \sm (U \cup \set{e}))|_0, S|_1]   + Y[(T \cup U) |_1 \cap (E' \sm U)|_0, S|_1]  \\
&=& Y[(T \cup U)|_1 \cap (E' \sm (U \cup \set{e})|_0, S|_1],
\end{eqnarray*}
where the last equality follows from ($\SA_+ 3$). Now if we apply this observation iteratively to every edge in $E'$, we see that
\begin{equation}\label{SA+MTup1a}
 \sum_{U \subseteq E'} (f_U)_0 =  \sum_{U \subseteq E'} Y[(T \cup U)|_1 \cap (E' \sm U)|_0, S|_1]  = Y[T|_1, S|_1].
\end{equation}
Then we can extend~\eqref{SA+MTup1a} to the other entries of $f_U$, and obtain
\begin{equation}\label{SA+MTup2}
\sum_{U \subseteq E'} f_U =\left( Y[T|_1, S|_1], Y[(T \cup \set{e_1})|_1, S|_1], \ldots, Y[(T \cup \set{e_{|E'|}}) |_1, S|_1] \right)^{\top},
\end{equation}
where $e_1, \ldots, e_{|E'|}$ are the edges in $E'$.

Moreover, observe that $f_U = \begin{pmatrix} (f_U)_0  \\  (f_U)_0 \chi^U \end{pmatrix}$ for all $U \subseteq E'$, and by ($\SA_+' 4$) we know that $(f_U)_0 > 0$ only if $U \cup T \cup S$ is a matching of $G$, which implies that $U$ is a matching contained in $E'$. Since $E'$ spans $2n-2i+1$ vertices, such a $U$ must have size at most $n-i$.
Thus, for each $f_U$ such that $(f_U)_0 >0$, we know that $\sum_{i \in E'} (f_U)_i \leq (n-i) (f_U)_0$. Therefore, by~\eqref{SA+MTup2},
\[
\binom{2n-2i+1}{2} \frac{Z_{i+1}}{|\M_{i+1}|} = \sum_{i \in E'} Y[(T \cup \set{e_i})|_1 , S|_1]  \leq (n-i) Y[T|_1, S|_1] = (n-i) \frac{Z_i}{|\M_{i}|}.
\]
Notice that
\[
|\M_{j}| = \frac{1}{j!} \binom{2n+1}{2}\binom{2n-1}{2} \cdots \binom{2n-2j+3}{2} = \frac{(2n+1)!}{2^j j! (2n-2j+1)!},
\]
for all $j \in \set{0, 1, \ldots, n}$. Thus, we obtain that
\[
Z_{i+1} \leq \frac{(n-i)|\M_{i+1}|}{\binom{2n-2i+1}{2} |\M_{i}|}Z_i = \frac{n-i}{i+1} Z_i.
\]
This concludes the proof, as we see that the facets of $MT(G)_I$ corresponding to smaller odd cliques in $G$ are also generated by $\SA_+'^{k}$.
\end{proof}

Recall that, as shown in~\cite{StephenT99a}, the $\LS_+$-rank of $MT(K_{2n+1})$ is exactly $n$. Thus, the techniques we proposed prove that $\SA_+'$ performs strictly better on this family of polytopes.

Next, we show that the notion of $\l$-established matrices can also be applied to provide an upper bound on the $\BZ_+'$-rank of  $MT(K_{2n+1})$.


\begin{tm}\label{BZ+MT}
The $\BZ'_+$-rank of $MT(K_{2n+1})$ is at most $\left\lceil \sqrt{2n + \frac{1}{4}} - \frac{1}{2} ~\right\rceil$.
\end{tm}

\begin{proof}
Let $G = K_{2n+1}$ and $P = MT(G)$. First, we show that every subset $W \subseteq E$ of size up to $\left\lfloor \frac{k+1}{2} \right\rfloor$ is a wall generated by $\BZ_+'^k$. Given any edge $\set{i,j} \in W$, take a vertex $v \not\in \set{i,j}$. Then $\set{\set{i,v}, \set{i,j}}$ and $\set{\set{j,v}, \set{i,j}}$ are both $k$-small obstructions for any $k \geq 1$, and their intersection contains $\set{i,j}$. If we do this for every edge in $W$, then we see that there is a set of at most $2|W| \leq k+1$ obstructions that generate $W$ as a wall.

Therefore, every set $S$ of size up to $k \left\lfloor \frac{k+1}{2} \right\rfloor$ is a tier, and the variable $S|_1$ is generated. Since $k \geq \left\lceil \sqrt{2n + \frac{1}{4}} - \frac{1}{2} ~\right\rceil $ implies $k \left\lfloor \frac{k+1}{2} \right\rfloor \geq n$, we see that every matrix $Y \in \tilde{\BZ}_+'^{k}(P)$ is $n$-established. By ($\BZ' 3$), $Y[S|_1, S'|_1] >0$ only if $S \cup S'$ is a matching, which implies $Z_{i} = 0,~\forall i > n$. Thus, we can apply \autoref{Ziupcor} and deduce that $Z_1 \leq n$. Therefore, $\sum_{e \in E} x_e \leq n$ is valid for $\BZ_+'^{k}(P)$.

Again, since the facets of $MT(G)_I$ corresponding to smaller odd cliques in $G$ are also generated by $\BZ_+'^{k}$, we are finished.
\end{proof}

The above upper bound also applies to the slightly weaker $\BZ_+$ operator. Also, we can show that the $\BZ_+'$-rank of $MT(K_{2n+1})$ is at least $\sqrt{n} -1 $. This relies on the fact that the $\SA'_+$-rank of $MT(K_{2n+1})$ is at least $\frac{n}{2}$, the detailed proof for which is rather substantial, and is planned for a subsequent publication.

\subsection{When strengthening $\SA$ by a PSD constraint does not help}

We have seen cases in which polyhedral operators and positive semidefinite operators do not gain any strength by lifting a given set to a higher dimension. Here, we show some instances in which adding a positive semidefiniteness constraint to a polyhedral operator does not help, extending a result by Goemans and the second author in~\cite{GoemansT01a}. In this section, we will use $v[i]$ to denote the $i^{\tn{th}}$ entry of a vector $v$. Given $x \in [0,1]^n$ and two disjoint sets of indices $I,J \subseteq [n]$, we define the vector $x^I_J \in [0,1]^n$ where
\[
x^I_J[i]:= \left\{
\begin{array}{ll}
1 & \tn{if $i \in I$;}\\
0 & \tn{if $i \in J$;}\\
x[i] & \tn{otherwise.}
\end{array}
\right.
\]

In other words, $x^I_J$ is the vector obtained from $x$ by setting all entries indexed by elements in $I$ to 1, and all entries indexed by elements in $J$ to 0. Then we have the following.

\begin{tm}\label{SAsharp}
Let $P \subseteq [0,1]^n$ and $x \in [0,1]^n$. If $x^I_J \in P$ for all $I,J \subseteq [n]$ such that $|I| + |J| \leq k$, then $x \in \SA_{+}^k(P)$.
\end{tm}

\begin{proof}
We prove our claim by constructing a matrix in $\mathbb{R}^{\A_k \times \A_k}$ that certifies $ x\in \SA_{+}^k(P)$. Recall that $\A_k = \set{ S|_1 \cap T|_0 : S,T \subseteq [n], S \cap T = \es, |S| + |T| \leq k}$ and $\A^+_k = \set{ S|_1 : |S| \leq k}$. For each $I \subseteq [n], |I| \leq k$, define $y^{(I)} \in \A^+_k$ such that
\[
y^{(I)}[S |_1]:= \left\{
\begin{array}{ll}
\prod_{i \in S \setminus I} x_i & \tn{if $I \subseteq S$;}\\
0 & \tn{otherwise.}
\end{array}
\right.
\]
Note that in the case of $y^{(I)}[I|_1]$, the empty product is defined to evaluate to 1.

Next, we define $Y \in \mathbb{R}^{\A^+_k \times \A^+_k}$ as
\[
Y := \sum_{S \subseteq [n], |S| \leq k} \left( \prod_{i \in S} x_i(1-x_i) \right) y^{(S)} (y^{(S)})^{\top}.
\]
Note that $Y \succeq 0$. Now given $S,T \subseteq [n], |S|, |T| \leq k$, observe that
\begin{eqnarray*}
Y[S|_1, T|_1] &=& \sum_{U \subseteq S \cap T} \left( \prod_{i \in U} x_i(1-x_i)\right) \left( \prod_{i \in S \setminus U} x_i \right)\left( \prod_{i \in T \setminus U} x_i \right)\\
&=& \left( \prod_{i \in S \cup T} x_i  \right)\left( \sum_{U \subseteq S \cap T}\left( \prod_{i \in U} (1-x_i) \right)\left( \prod_{i \in (S \cap T) \setminus U} x_i \right)\right) \\
&=& \prod_{i \in S \cup T} x_i.
\end{eqnarray*}
Next, define $U \in \mathbb{R}^{\A_k \times \A_k^+}$ such that
\[
U^{\top} ( e_{S|_1 \cap T|_0}) := \sum_{W : S \subseteq W \subseteq (S \cup T)} (-1)^{|W \setminus S|} e_{W|_1},
\]
for all disjoint $S,T \subseteq [n]$ such that $|S| + |T| \leq k$. Now consider the matrix $Y' := UYU^{\top}$. Then given $\a, \b \in \A_k$ where $\a = S|_1 \cap T|_0$ and $\b = S'|_1 \cap T'|_0$,
\begin{equation}\label{SAsharp0}
Y'[\a, \b] =  \left(  U^{\top}  e_{\a} \right)^{\top} Y  \left( U^{\top}  e_{\b} \right)  = \sum_{W \subseteq T \cup T'} (-1)^{|W|} \left( \prod_{i \in (S \cup S') \cup W} x_i \right).
\end{equation}
Now if $\a \cap \b = \es$, then there exists an index $\l \in [n]$ where $\l \in (S \cup S') \cap (T \cup T')$. In this case, $Y'[\a, \b]$ evaluates to
\[
\sum_{W' \subseteq ((T \cup T') \sm \set{\l} )} \left( (-1)^{|W'|} \left( \prod_{i \in (S \cup S') \cup W'} x_i \right) + (-1)^{|W' \cup \set{\l}|} \left( \prod_{i \in (S \cup S') \cup (W' \cup \set{\l})} x_i \right) \right).
\]
The latter expression leads to
\begin{eqnarray}
\sum_{W' \subseteq (T \cup T') \sm \set{\l} )} \left( (-1)^{|W'|} \left( \prod_{i \in (S \cup S') \cup W'} x_i \right) + (-1)^{|W'| + 1} \left( \prod_{i \in (S \cup S') \cup W'} x_i \right) \right)=0.
\label{SAsharp1a}
\end{eqnarray}
Now, if $\a \cap \b \neq \es$, then $S_1 \cup S_2$ and $T_1 \cup T_2$ are disjoint, and we obtain that
\begin{equation}\label{SAsharp1b}
Y'[\a, \b] = \left( \prod_{i \in S \cup S'} x_i \right) \left(\sum_{W \subseteq T \cup T'} (-1)^{|W|}   \left( \prod_{i \in W} x_i \right) \right)  = \left( \prod_{i \in S \cup S'} x_i \right) \left( \prod_{i \in T \cup T'} (1-x_i) \right).
\end{equation}
We claim that $Y' \in \tilde{\SA}_+^k(P)$. First, notice that $Y'[\F, \F] = 1$, so $(\SA_+ 1)$ holds. Next, given $\a = S|_1 \cap T|_0 \in \A_k$,
\[
\het{x}\left( Y' e_{\a} \right) = \begin{pmatrix} Y'[\F, \a] \\ Y'[\F,\a] x^S_T \end{pmatrix} \in K(P),
\]
where we applied the assumption that $x^S_T \in P$. We also see from~\eqref{SAsharp1a} and~\eqref{SAsharp1b} that $Y'' \geq 0$, and so ($\SA_+ 2$) is satisfied. It is also easy to verify from~\eqref{SAsharp1a} and~\eqref{SAsharp1b} that ($\SA_+ 3$), ($\SA_+ 4$) and ($\SA_+ 5$) hold as well. Also, $Y \succeq 0$ implies $Y' \succeq 0$. Therefore, since $\het{x}(Y' e_{\F}) = \het{x}$, it follows that $x \in \SA_+^k(P)$.
\end{proof}

From the above, we are able to characterize some convex sets for which $\SA_{+}^k$ does not produce a tighter relaxation than an operator as weak as $\LS_0^k$.

\begin{cor}\label{SAsharpcor}
Suppose $P \subseteq [0,1]^n$ is a convex set such that, for all $x \in P$ and for all $I,J, I',J' \subseteq [n]$ such that $I \cup J = I' \cup J'$ and $|I| + |J| = k$,
\[
x^I_J \in P \iff x^{I'}_{J'} \in P.
\]
Then
\[
\SA_+^k(P) = \LS_0^k(P) = \bigcap_{I \subseteq [n], |I| = k} \set{x :  x^{I}_{\es} \in P}.
\]
\end{cor}

The two results above generalize Theorem 4.1 and Corollary 4.2 in~\cite{GoemansT01a}, respectively. Since $\SA_{+}$ dominates both $\LS_+$ and $\SA$, \autoref{SAsharpcor} immediately implies the following:

\begin{cor}\label{SA+rankneg}
Given $p \in \mathbb{R}$, let
\[
P(p) := \set{x \in [0,1]^n : \sum_{i \in S} x_i + \sum_{i \not\in S} (1-x_i) \leq n- \frac{p+1}{2},~\forall S \subseteq [n]}.
\]
Then $\SA_{+}^k(P(0)) = P(k)$, for all $k \in \set{0,1,\ldots,n}$. In particular, the $\SA_{+}$-rank of $P(0)$ is $n$.
\end{cor}

One can apply the same argument used in \autoref{SASA'} to show that $\SA_{+}^{2k}(P) \subseteq \SA_{+}'^{k}(P)$ in general. Thus, the $\SA'_{+}$-rank of $P(0)$ is at least $\left\lceil \frac{n}{2} \right\rceil$. On the other hand, the proof of \autoref{BZ''upper} (given in the Appendix)
can be adapted to show that the $\SA_+'$-rank of any polytope contained in $[0,1]^n$ is at most $\left\lceil \frac{n+1}{2} \right\rceil$. Thus, we see that in this case, $\SA_+'$ requires roughly $\frac{n}{2}$ fewer rounds than $\SA_+$ to show that $P(0)$ has an empty integer hull.

It was shown in~\cite{BienstockZ04a} that $\BZ^2(P(0)) = \es = P(0)_I$ (implying
$\BZ'^{2}(P(0)) = \BZ_+'^{2}(P(0)) = \es$). However, since the run-time of $\BZ$ depends on the size of the system of inequalities describing $P$ (which in this case is exponential in $n$), the relaxation generated by $\BZ^2$ is not tractable. In contrast, note that it is easy to find an efficient separation oracle for $P(0)$ (e.g. by observing $x \in P(0)$ if and only if $\sum_{i=1}^n | x_i - \frac{1}{2} | \leq \frac{n-1}{2}$), and thus one could optimize a linear function over, say, $\SA^k(P(0))$ in polynomial time for any $k = O(1)$. The reader may refer to \autoref{fig0} for a complete classification of operators that depend on the algebraic description of the input set $P$, as opposed to those that only require a weak separation oracle.

\section{Integrality gaps of lift-and-project relaxations}\label{sectionintgap}

So far, we have been using the rank of a relaxation with respect to a lift-and-project operator as the measure of how far that relaxation is away from its integer hull. Another measure of the ``tightness'' of a relaxation that is commonly used and well studied is the integrality gap. Again, let $P \subseteq [0,1]^n$ be a convex set such that $P_I \neq \es$, and suppose $c \in \mathbb{R}^{n}$. Then
\[
\gamma_c(P) := \frac{ \max\set{ c^{\top}x : x \in P}}{ \max\set{ c^{\top}x : x \in P_I}}
\]
is the \emph{integrality gap} of $P$ with respect to $c$.
Observe that, given $P, P'$ such that $P_I = P'_I$ and $P \subseteq P'$, then $\gamma_c(P) \leq \gamma_c(P')$ for all $c$. Thus, our earlier results immediately imply the following:

\begin{cor}\label{uselesscor}
Suppose $P \subseteq [0,1]^n$, and two lift-and-project operators $\Gamma_1, \Gamma_2$ satisfy the conditions in either \autoref{OMCuseless} or~\autoref{RCMCuseless}. Then
\[
\gamma_c(\Gamma_1(P)) \leq \gamma_c (\Gamma_2(P)),
\]
for all $c \in \mathbb{R}^n$.
\end{cor}

Next, we present another approach for obtaining an integrality gap result. Since in many optimization problems we are interested in computing the largest or smallest cardinality of a set among a given collection (e.g. the stable set problem and the max-cut problem), we are often optimizing in the direction of $\bar{e}$. Moreover, we have seen that many hardness results have been achieved by highly symmetric combinatorial objects (e.g. the complete graph), which correspond to polytopes that have a lot of symmetries. These symmetries can significantly simplify the analyses of lift-and-project relaxations. For instance, they could allow us to assume that there are certificate matrices in the lifted space with very few distinct entries.

The idea of using symmetry and convexity to reduce the number of parameters involved in a problem instance have been widely exploited in both computational work and theoretical research. This at least goes back to Lov{\'a}sz's seminal work on the theta function in~\cite{Lovasz79a} and related findings by Schrijver in~\cite{Schrijver79a}. Also during the 1970s, Godsil used similar ideas in his work in algebraic graph theory (see~\cite{ChanG97a} for a more recent survey). More recently, these ideas have also been proven useful in reducing SDP instances~\cite{GatermannP04a, deKlerkPS07a}, bounding the crossing number of graphs~\cite{deKlerkMPRS06a}, and obtaining SDP relaxations for polynomial optimization problems~\cite{MuramatsuWT13a}. Thus, the following ideas have been useful in the past and could continue to be useful.

We say that a compact convex set $P \subseteq [0,1]^n$ is \emph{symmetric} if there exists an $n$-by-$n$ permutation matrix $Q$ such that $\set{ Qx : x \in P} \subseteq P$, with the condition that the permutation on $[n]$ corresponding to $Q$ has no cycles of length smaller than $n$. Note that the reverse containment is implied by the definition, as $Q^n = I$. Moreover, observe that if $P$ is symmetric, so is $P_I$.

Next, we say that a lift-and-project operator $\Gamma$ is \emph{symmetry preserving} if given any symmetric, compact convex set $P$, $\Gamma(P)$ is also symmetric, compact and convex. All named operators mentioned in this paper are symmetry preserving. (In the case when $\Gamma$ is one of the Bienstock--Zuckerberg variants, a symmetric algebraic description of $P$ is required.) Then we have the following:

\begin{tm}\label{intgap1}
Let $P \subseteq [0,1]^n$ be a symmetric, compact and convex set, and let $\Gamma$ be a symmetry preserving operator. Then, the integrality gaps of $\gamma_{\bar{e}}(\Gamma(P))$ are attained by a nonnegative multiple of $\bar{e}$.
\end{tm}

\begin{proof}
First, we show that for any $y \in \Gamma(P), \left( \frac{y^{\top} \bar{e}}{n} \right) \bar{e} \in P$. Let $Q$ be a permutation matrix that certifies the symmetry of $P$. Then given $y \in \Gamma(P)$, we know that $y, Qy, \ldots, Q^{n-1}y \in \Gamma(P)$, as $\Gamma$ preserves symmetry.  Since $Q$ essentially permutes the $n$ coordinates of $P$ around in an $n$-cycle, we know that $\sum_{i=0}^{n-1} Q^i = J$, the all-ones matrix. By the convexity of $\Gamma(P)$,
\[
\sum_{i=0}^{n-1} \frac{1}{n}  \left( Q^i y \right) = \frac{1}{n} J y = \left( \frac{y^{\top} \bar{e}}{n} \right) \bar{e} \in \Gamma(P).
\]
Now if $y$ is a point that attains the maximum integrality gap in the direction of $\bar{e}$, then we could use the above construction to obtain a multiple of $\bar{e}$ that achieves the same objective value. Hence, our claim follows.
\end{proof}

Note that \autoref{intgap1} immediately implies the following:

\begin{cor}\label{intgap2}
Suppose $P \subseteq [0,1]^n$ is a symmetric, compact and convex set, and $\Gamma$ is a symmetry preserving operator. If $\bar{e}^{\top}x < \l$ is valid for $\Gamma(P)$ and $P_I \neq \es$, then
\[
\gamma_{\bar{e}}(\Gamma(P)) < \frac{ \l}{\max \set{\sum_{i=1}^n x_i : x \in P_I}}.
\]
\end{cor}

Of course, the $\gamma_{-\bar{e}}$ analogs of \autoref{intgap1} and \autoref{intgap2} can be obtained by essentially the same observations. Thus, we see that in many cases, it suffices to check whether a certain multiple of $\bar{e}$ belongs to $\Gamma(P)$ to obtain a bound on $\gamma_{\bar{e}}(\Gamma(P))$. This structure, when present, makes the analysis significantly easier, as often times we can apply the above symmetry-convexity argument to the
certificate matrices in $\tilde{\Gamma}(P)$ as well, and identify many of the variables in the lifted space.


\bibliographystyle{alpha}
\bibliography{ref}

\appendix

\section{The Original $\BZ$ Operator}

In this section, we state the original $\BZ$ operator in our unifying
language, and show that it is dominated by $\BZ'$.

The refinement step of $\BZ^k$ coincides with $\BZ'^k$ --- both operators derive $k$-small obstructions from the linear inequalities describing $P$, and use them to construct $\O_k(P)$.
Then $\BZ^k$ defines its set of walls to be
\[
\W_k := \set{ \bigcup_{i,j \in [\l],  i\neq j} (O_i \cap O_j) : O_1,\ldots O_{\l} \in \O_k, \l \leq k+1}.
\]
Note that unlike for $\BZ'^k$, $\BZ^k$ does not guarantee that the singleton sets are walls, and we will see that this could make a difference in performance. As for the tiers, $\BZ^k$ defines them to be the sets of indices that can be written as the union of up to $k$ walls in $\W_k$. Thus, $\BZ^k$ only generates a polynomial size subset of the tiers used in $\BZ'^k$. Then the lifting step of $\BZ^k$ (and $\BZ_+^k$) can be described as follows:

\begin{enumerate}
\item
Define $\A'$ to be the set consisting of the following:
\begin{itemize}
\item
$\F$ and $i|_1, i|_0,~\forall i \in [n]$.
\item
Suppose $S := \bigcup_{i=1}^{\l} W_i$ is a tier. Then we do the following:
\begin{itemize}
\item
For each $\l$-tuple of sets, $(T_1, \ldots, T_{\l})$ such that $T_i \subseteq W_i,~\forall i \in [\l]$ and \\ $\sum_{i=1}^{\l} |T_i| \leq k$, include the set
\begin{equation}\label{BZvar1}
 \left. \left( \bigcup_{i=1}^{\l} W_i \sm T_i \right) \right|_1 \cap \left. \left( \bigcup_{i=1}^{\l} T_i \right) \right|_0.
\end{equation}
If $\sum_{i=1}^{\l} |T_i| = k$ and $T_{\l} \subset W_{\l}$, then include the set
\begin{equation}\label{BZvar2}
 \left. \left( \bigcup_{i=1}^{\l-1} W_i \sm T_i \right) \right|_1 \cap \left. \left( \bigcup_{i=1}^{\l-1} T_i \right) \right|_0
 \cap W_{\l}|_{< |W_{\l}| - |T_{\l}|}.
\end{equation}
\end{itemize}
\end{itemize}
\item
Let $\tilde{\BZ}^k(P)$ denote the set of matrices $Y \in \mathbb{S}^{\A'}$ that satisfy all of the following conditions:
\begin{itemize}
\item[($\BZ1$)]
$Y[\F, \F] = 1$.
\item[($\BZ2$)]
For any column $x$ of the matrix $Y$,
\begin{itemize}
\item[(i)]
$0 \leq x_{\alpha} \leq x_{\F}$, for all $\a \in \A'$.
\item[(ii)]
$\het{x}(x) \in K(\O_k(P))$.
\item[(iii)]
$x_{i|_1} + x_{i|_0} = x_{\F}$, for every $i \in [n]$.
\item[(iv)]
For each $\a \in \A'$ of the form $S|_1 \cap T|_0$, impose the inequalities
\begin{eqnarray}
\label{Ck41o} x_{i|_1} &\geq& x_{\a}, \quad \forall i \in S; \\
\label{Ck42o} x_{i|_0} &\geq& x_{\a}, \quad \forall i \in T; \\
\label{Ck43o} \sum_{i \in S} x_{i|_1} + \sum_{i \in T} x_{i|_0} - x_{\alpha} &\leq& (|S| + |T| -1)x_{\F}.
\end{eqnarray}
\item[(v)]
For each $\a \in \A'$ of the form $S|_1 \cap T|_0 \cap U|_{< r}$, impose the inequalities
\begin{eqnarray}
\label{Ck44o} x_{i|_1} &\geq& x_{\a}, \quad \forall i \in S; \\
\label{Ck45o} x_{i|_0} &\geq& x_{\a}, \quad \forall i \in T; \\
\label{Ck46o} \sum_{i \in U} x_{i|_0} &\geq& (|U| - (r-1)) x_{\alpha}.
\end{eqnarray}
\item[(vi)]
For each variable of the form~\eqref{BZvar1}, if $|W_{\l}| + \sum_{i=1}^{\l-1} |T_i| \leq k$, impose
\begin{eqnarray}
 \nonumber && \sum_{U \subseteq W_{\l}}
 x_{ \left( \bigcup_{i=1}^{\l-1} W_i \sm T_i \right)|_1 \cap \left( \bigcup_{i=1}^{\l-1} T_i \right)|_0 \cap (W_{\l} \sm U)|_1 \cap U|_0} \\
 \label{sumwall1o}&=&x_{ \left( \bigcup_{i=1}^{\l-1} W_i \sm T_i \right)|_1 \cap \left( \bigcup_{i=1}^{\l-1} T_i \right)|_0}.
\end{eqnarray}
Otherwise, define $r:= k - ( \sum_{i=1}^{\l-1} |T_i|)$, and impose
\begin{eqnarray}
 \nonumber && \sum_{U \subseteq W_{\l}, |U| \leq r}
 x_{ \left( \bigcup_{i=1}^{\l-1} W_i \sm T_i \right)|_1 \cap \left( \bigcup_{i=1}^{\l-1} T_i \right)|_0 \cap (W_{\l} \sm U)|_1 \cap U|_0} \\
 \nonumber &+& x_{ \left( \bigcup_{i=1}^{\l-1} W_i \sm T_i \right)|_1 \cap \left( \bigcup_{i=1}^{\l-1} T_i \right)|_0 \cap W_{\l}|_{< |W_{\l}| - r}} \\
 \label{sumwall2o} &=& x_{ \left( \bigcup_{i=1}^{\l-1} W_i \sm T_i \right)|_1 \cap \left( \bigcup_{i=1}^{\l-1} T_i \right)|_0}.
\end{eqnarray}
\end{itemize}
\item[($\BZ3$)]
For all $\a, \b \in \A'$ such that $\a \cap \b = \es$, or $\a \cap \b$ is contained in $O|_1$ for some $k$-small obstruction $O \in \O_k$, $Y[\a,\b] = 0$.
\item[($\BZ4$)]
For all $\a_1,\b_1, \a_2,\b_2 \in \A'$ such that $\a_1 \cap \b_1 = \a_2 \cap \b_2$, $Y[\a_1, \b_1] = Y[\a_2, \b_2]$.
\end{itemize}
\item
Define
\[
\BZ^k(P) := \set{x \in \mathbb{R}^n: \exists Y \in \tilde{\BZ}^k(P),  \het{x}(Ye_{\F})= \het{x}},
\]
and
\[
\BZ_+^k(P) := \set{x \in \mathbb{R}^n: \exists Y \in \tilde{\BZ}_+^k(P),  \het{x}(Ye_{\F})= \het{x}},
\]
where $\tilde{\BZ}_+^k(P) := \tilde{\BZ}^k(P) \cap \mathbb{S}_+^{\A'}$.
\end{enumerate}

In~\cite{BienstockZ04a}, $\BZ$ was defined so that the first relaxation in the hierarchy is $\BZ^2(P)$, with $\BZ^{n+1}(P)$ being the $n^{\tn{th}}$ relaxation that is guaranteed to be $P_I$. We have modified their definitions and presented their operators such that the relaxations are instead $\BZ^1(P), \ldots, \BZ^n(P)$, to align them with the other named operators mentioned in this manuscript.

\section{Relationships among Variants of the $\BZ$ Operator,\\
and some omitted Proofs}

Next, we show that $\BZ'$ and $\BZ_{+}'$ indeed dominate their original counterparts.

\begin{prop}
For every polytope $P \subseteq [0,1]^n$ and integer $k \geq 1$, $\BZ'^k(P) \subseteq \BZ^k(P)$ and
$\BZ_{+}'^{k}(P) \subseteq \BZ_+^k(P)$.
\end{prop}

\begin{proof}
It is apparent that every variable generated by $\BZ^k$ is also generated by $\BZ'^k$. The only nontrivial case is when $\BZ^k$ generates a variable of the form
\begin{equation}\label{BZ'BZ1}
 \left. \left( \bigcup_{i=1}^{\l-1} W_i \sm T_i \right) \right|_1 \cap \left. \left( \bigcup_{i=1}^{\l-1} T_i \right) \right|_0
 \cap W_{\l}|_{< |W_{\l}| - |T_{\l}|}
 \end{equation}
 such that $W_{\l}$ is not disjoint from $\bigcup_{i=1}^{\l-1} W_i \sm T_i $. In this case if we define $W' := W_{\l} \sm \left( \bigcup_{i=1}^{\l-1} W_i \sm T_i \right) $, then the above is equivalent to $\es$ if $|W'| \leq |T_{\l}|$, and
\[
\left. \left( \bigcup_{i=1}^{\l-1} W_i \sm T_i \right) \right|_1 \cap \left. \left( \bigcup_{i=1}^{\l-1} T_i \right) \right|_0
 \cap W'|_{< |W'| - |T_{\l}|}
\]
otherwise, which we know is generated by $\BZ'^k$. Also, note that in the case of $\bigcup_{i=1}^{\l-1} W_i \sm T_i $ and $\bigcup_{i=1}^{\l-1} T_i$ having a nonempty intersection, \eqref{BZ'BZ1} evaluates to the empty set.

Also, the condition ($\BZ'3$) is more easily triggered than ($\BZ 3$), and thus $\BZ'$ forces more variables to be zero and is more restrictive. It is also not hard to see that the constraints~\eqref{Ck41}--\eqref{sumwall2} imply their corresponding counterparts~\eqref{Ck41o}--\eqref{sumwall2o} in $\BZ$. Hence, we have $\tilde{\BZ}'^k(P) \subseteq \tilde{\BZ}^k(P)$, and it follows readily that $\BZ'^k(P) \subseteq \BZ^k(P)$ and $\BZ_+'^k(P) \subseteq \BZ_+^k(P)$.
\end{proof}

As Bienstock and Zuckerberg proved in~\cite{BienstockZ04a}, the original $\BZ$ operator can efficiently solve many set covering type problems
which require exponential effort to solve by previously used operators such as $\SA$. However, since $\BZ^k$ does not ensure that it generates walls of small sizes, its tiers (which are unions of walls) could all be large, and the lifted set of variables $\A'$ does not necessarily contain $\A_{k}$ as in $\BZ'^k$. In fact, in some cases, $\BZ^k$ performs no better than one round of $\LS$.

\begin{prop}\label{LSownsBZ}
Let $p,q$ be positive integers such that $1 \leq q < p$, and let
\[
P := \set{ x \in [0,1]^{p} : \sum_{i=1}^{p} x_i \leq  q+ \frac{1}{2}}.
\]
If $(k+1)(k+2) \leq p-q$ and $k+1 \leq q$, then $\BZ^k(P) = \LS(P)$ and $\BZ_+^k(P) = \LS_+(P)$.
\end{prop}

\begin{proof}
Since $q + \frac{1}{2} > k+1$, there are no $k$-small obstructions of size $k+1$ or less. Thus, $S \subseteq [n]$ is a $k$-small obstruction if and only if $|S| \geq p-(k+1)$, which implies that every wall (and hence, every tier) has size at least $p-(k+1)^2$. If $p - (k+1)^2 -(k+1) \geq q$, then we see that every tier is $P$-useless. The only remaining non-useless variables are $\F, i|_1$ and $i|_0$ for all $i \in [n]$. Thus, $\BZ^k(P) = \LS(\O_k(P))$ and $\BZ_+^k(P) = \LS_+(\O_k(P))$.

Furthermore, $\O_k(P) = P$ whenever $k+1 \leq p-q$, which is implied by $(k+1)(k+2) \leq p-q$. Thus, our claim follows.
\end{proof}

Since $\LS(P) \subset P$ whenever $P \neq P_I$, the above implies that one can construct examples in which $\LS^2(P) \subset \BZ^k(P)$ for arbitrarily large $k$. On the other hand, it is easy to obtain a lift-and-project operator that has the unique strength of $\BZ$, while also refining the earlier operators (for instance, by simply taking $\Gamma^k(P) = \SA^k(P) \cap \BZ^k(P)$).

We can take this one step further. Recall that $\BZ'$ generates exponentially many variables in its lifted space, and thus does not admit a straightforward polynomial-time implementation. However, the number of variables generated becomes polynomial in $n$ if we instead use the original $\BZ$'s rule of generating tiers (i.e., defining $S$ to be a tier if it is a union of up to $k$ walls). Let $\BZ''$ denote this new operator. Then $\BZ''$ is just like the original $\BZ$, except it has polynomially more variables, always ensures the singleton sets are walls, and imposes the condition ($\BZ' 3$) instead of the weaker ($\BZ 3$). Also, just like ($\SA' 4$) and ($\SA_+' 4$), the condition ($\BZ'3$) can be efficiently verified, given we have an efficient separation oracle for $P$, and the condition is only checked polynomially many times. Replacing ($\BZ 3$) with ($\BZ' 3$) boasts the advantage of eliminating the operator's dependence on the set of obstructions in the lifting step, and allows us state the operator as a two-step process. Thus. if $k=O(1)$ and we have a compact description of $P$, then $\BZ''^k(P)$ is tractable. It is also not hard to see that $\BZ''$ dominates both $\SA'$ and $\BZ$. Moreover, the following is true:

\begin{prop}\label{BZ''upper}
The $\BZ''$-rank of $P$ is at most $\left\lceil \frac{n+1}{2} \right\rceil$, for all $P \subseteq [0,1]^n$.
\end{prop}

\begin{proof}
Let $Y \in \tilde{\BZ}''^k(P)$ such that $k \geq \frac{n+1}{2}$. We show that $\het{x}(Ye_{\F}) \in K(P_I)$. Notice that $\BZ''^k$ generates $S := [k]$ as a tier (derived from $k$ singleton-set walls), and we know by~\eqref{sumwall1} and the symmetry of $Y$ that
\begin{equation}\label{SSbar2}
Ye_{\F} = \sum_{T \subseteq S} Ye_{T|_1 \cap (S \sm T)|_0}.
\end{equation}
In the remainder of this proof, we let $Y_T$ denote $Ye_{T|_1 \cap (S \sm T)|_0}$ to reduce cluttering. Note that since $|S| = k$, $\BZ''^k$ does generate the variable $T|_1 \cap (S \sm T)|_0$ for all $T \subseteq S$, and so $Y_T$ is well defined.

Next, we prove that $\het{x}(Y_T) \in K(P_I)$ for every $T \subseteq S$. Then by~\eqref{SSbar2}, it follows that $\het{x} (Ye_{\F}) \in K(P_I)$.
For convenience, we let $\bar{S}$ denote $[n] \setminus S$. Notice that
\begin{equation}\label{prop211}
(Y_T)_{\F} = \sum_{S' \subseteq \bar{S}} (Y_T)_{S'|_1 \cap (\bar{S} \sm S')|_0}
\end{equation}
by~\eqref{sumwall1}. Also, since $k \geq \frac{n+1}{2}$, $|\bar{S}| = n - k \leq k-1$. Hence, $\set{j} \cup \bar{S}$ is a tier for all $j \in [n]$, and
\begin{equation}\label{prop212}
(Y_T)_{j|_1} = \sum_{S' \subseteq \bar{S}} (Y_T)_{(j \cup S')|_1 \cap (\bar{S} \sm S')|_0}, \quad \forall j \in [n].
\end{equation}
Next, for all $T' \subseteq \bar{S}$, we define $Y_{T, T'} \in \mathbb{R}^{n+1}$ such that
\[
(Y_{T,T'})_i =\left\{
\begin{array}{ll}
(Y_T)_{T'|_1 \cap (\bar{S} \sm T')|_0} & \tn{if $i=0$ or $i \in T \cup T'$;}\\
0 & \tn{otherwise.}
\end{array}
\right.
\]
From~\eqref{prop211}, \eqref{prop212}, and the construction of $Y_{T, T'}$, we obtain that
\[
\het{x}(Y_T) = \sum_{T' \subseteq \bar{S}} Y_{T,T'},~\forall T \subseteq S.
\]
Thus, it suffices to show that $Y_{T,T'} \in K(P_I),~\forall T \subseteq S, T' \subseteq \bar{S}$. This is obviously true if $(Y_{T,T'})_0 = 0$. If $(Y_{T,T'})_0 > 0$, then by ($\BZ'3$) we know that $(T \cup T')|_1 \cap ([n] \sm (T \cup T'))|_0 \cap P \neq \es$. Since $Y_{T,T'} = \begin{pmatrix} (Y_{T,T'})_0 \\ (Y_{T,T'})_0 \chi^{T \cup T'} \end{pmatrix}$, it follows that $Y_{T,T'} \in K(P_I)$, completing the proof.
\end{proof}

Likewise, we can define $\BZ_+''$ to be the positive semidefinite counterpart of $\BZ''$, and obtain a tractable operator that dominates both $\SA'_+$ and $\BZ_+$. Therefore, it follows that the $\BZ_{+}''$-rank of any $P \subseteq [0,1]^n$ is also at most $\left\lceil \frac{n+1}{2} \right\rceil$. Moreover, observe that the essential ingredients used in the above proof are the presence of the variables in $\A_{\left\lceil n+1/2 \right\rceil}$ in the lifted space and the condition ($\BZ' 3$), which also applies for the $\SA_+'^k$ relaxation for any $k \geq \frac{n+1}{2}$. Thus, the above proof can be slightly modified to show that the $\SA_+'$-rank of any polytope contained in $[0,1]^n$ is at most $\left\lceil \frac{n+1}{2} \right\rceil$. In contrast, we have seen in \autoref{SA+rankneg} an example in which the $\SA_+$-rank is $n$.

Since $\BZ''$ dominates $\LS$, we can deduce from~\autoref{LSownsBZ} that there are examples where $\BZ''^2(P) \subset \BZ^2(P)$. Next, we provide another instance in which $\BZ''$ outperforms $\BZ$.

\begin{prop}\label{newBZ}
Let $P:= \set{x \in [0,1]^7 : \sum_{i=1}^7 2x_i \leq 7}$. Then
\[
y := \left( 0.76, 0.76, 0.76, 0.3, 0.3, 0.3, 0.3 \right)^{\top} \in \BZ(P) \setminus \BZ''(P).
\]
\end{prop}

\begin{proof}
First, it is easy to see that $P_I = \set{ x \in [0,1]^7 : \sum_{i=1}^7 x_i \leq 3}$. Also, the $1$-small obstructions of $P$ is the collection of subsets of $[7]$ of size at least $5$, and it is not hard to see that $\O_1(P) = P$.

We first show that $\BZ''$ cuts off $y$. Since each wall is an intersection of up to two obstructions, every subset of $[7]$ of size between $3$ and $5$ is a wall. These sets are also exactly the tiers, as every tier consists of one wall in $\BZ''$. Suppose for a contradiction that there exists a certificate matrix $Y \in \tilde{\BZ}''(P)$ for $y$.  Consider the tier $S := \set{1,2,3}$. By~\eqref{sumwall2}, we know that
\begin{equation}\label{strrule}
Ye_{\F} = Ye_{S|_1} + \sum_{i \in S} Y e_{ (S \sm \set{i})|_1 \cap i|_0} + Y e_{S|_{<2}}.
\end{equation}
Since $\het{x} (Ye_{\a}) \in K(\O_1(P)) = K(P)$ for all variables $\a \in \A'$, we know from~\eqref{strrule} we can write  $\het{x} (Ye_{\F})$ as $z+w$, where $z := \het{x}(Ye_{S|_1})$, and $w \in K(P)$.

Now, applying~\eqref{Ck43} of $S|_1$ on the column $Y e_{\F}$, we obtain that
\[
Y[1|_1, \F] + Y[2|_1, \F] + Y[3|_1, \F] - Y[S|_1, \F] \leq (|S| - 1) Y[\F, \F].
\]
Hence, $z_0 = Y[\F, S|_1] = Y[S|_1, \F] \geq 3(0.76)-2 = 0.28$, and $w_0 = 1-z_0 \geq 0.72$. We also know that $\sum_{i=1}^7 w_i \leq \frac{7}{2} w_0$ (as $w \in K(P)$).

For $j \in \set{4,5,6,7}$, since $\conv(j|_1) \cap  \conv(S|_1) \cap P = \es$, our strengthened rule ($\BZ' 3$) requires that $Y[j|_1, S|_1] = 0$ (this is what sets $\BZ''$ apart from $\BZ$ in this example). Therefore, we have
\[
\sum_{i=1}^7 z_i = \sum_{i=1}^7 Y[i|_1, S|_1] \leq 3 Y[\F, S|_1] = 3 z_0.
\]
This would imply that the inequality
\[
\sum_{i=1}^7 x_i = \sum_{i=1}^7 (z_i + w_i) \leq 3z_0 + \frac{7}{2} w_0 \leq 3(0.28)+ \frac{7}{2}(0.72) = 3.36,
\]
is valid for $\BZ''(P)$, which is a contradiction as $\sum_{i=1}^7 y_i = 3.48$. Hence, $y \not\in \BZ''(P)$.

Finally, it can be checked computationally that $y \in \BZ(P)$. This finishes the proof of our claim.
\end{proof}

Note that the system of inequalities describing $\BZ(P)$ is already pretty large even for an example as small as that in~\autoref{newBZ}. Therein, any subset of $[7]$ of size between $3$ and $6$ can be expressed as the intersection of two $1$-small obstructions; so, each of them is a wall (and hence a tier). For each of these tiers $S$, there are $|S|+2$ associating variables ($S|_1, (S\sm \set{i})|_1 \cap i|_0$ for all $i \in S$, and $S|_{< |S|-2}$). Thus,
we see that $\tilde{\BZ}(P)$ is a subset of $603$-by-$603$ matrices, and our straightforward formulation of $\BZ(P)$ has more than two million constraints.

Next, we remark that, in general, adding redundant inequalities to the system $Ax \leq b$ could generate more obstructions and walls, and thus can improve the performance of $\BZ$ (and its variants). An example of this phenomenon is the following:

\begin{prop}\label{redundantineq}
Let $G$ be the graph in \autoref{fig3}. Furthermore, let $P$ be the set defined by the facets of $FRAC(G)$ and $P'$ be the system $P$ with the additional (redundant) inequality
\[
\sum_{i =1}^6 x_i \leq 3.
\]
Then
\[
\BZ_{+}'(P) \supset \BZ(P') = P_I.
\]
\end{prop}

\begin{figure}[htb]
 \begin{center}

\begin{tikzpicture}[scale =1 ,>=stealth',shorten >=1pt,auto,node distance=1cm,
  thick,main node/.style={circle, draw,font=\scriptsize\sffamily}]

  \node[main node] at (2,2) (1) {1};
  \node[main node] at (0,3) (2) {2};
  \node[main node] at (2,4) (3) {3};
  \node[main node] at (4,3) (4) {4};
  \node[main node] at (3,1) (5) {5};
  \node[main node] at (1,1) (6) {6};

 \draw (1) -- (2) ;
 \draw (3) -- (2) ;
 \draw (3) -- (4) ;
 \draw (4) -- (5) ;
 \draw (5) -- (6) ;
 \draw (1) -- (3) ;
 \draw (6) -- (2) ;
 \draw (1) -- (4) ;

\end{tikzpicture}

 \end{center}
\caption{A graph for which $\BZ$ performs better on $FRAC(G)$ with a redundant inequality.}\label{fig3}
\end{figure}
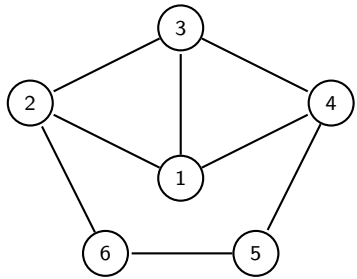

\begin{proof}
For the first claim, notice that the obstructions generated by $\BZ_{+}'$ are exactly the edge sets, so $\O_k(P) = (P)$. This also implies that all walls and tiers have size $1$, so
\[
\BZ_{+}'(P) = \LS_+(\O_k(P)) = \LS_+(P) \neq P_I,
\]
as it is shown in~\cite{LiptakT03a} that $P$ has $\LS_+$-rank $2$.

For the second claim, notice that with the additional inequality in $P'$, all sets of size at least $4$ are $1$-small obstructions, and thus all sets of size $2$ are walls (and hence tiers). In this case, $\BZ(P') \subseteq \SA^2(P') = P_I$.
\end{proof}

In fact, since $\BZ$ (and its variants) depends heavily on the algebraic description of the input set, it does not share some of the more fundamental properties with the earlier lift-and-project operators. For example, all other named operators mentioned in this paper preserves containment (i.e. $P \subseteq P'$ implies $\Gamma(P) \subseteq \Gamma(P')$). We give an example where that is not the case for $\BZ$.

\begin{prop}\label{containmentfail}
Let $G$ be the graph in \autoref{fig5}, and let $P$ be the set defined by the facets of $FRAC(G)$. Moreover, let $P'$ be the system as described in~\autoref{redundantineq}. Then
\[
P \subset P' \quad \tn{and} \quad \BZ(P) \not\subseteq \BZ(P').
\]
\end{prop}

\begin{figure}[htb]
 \begin{center}

\begin{tikzpicture}[scale =1 ,>=stealth',shorten >=1pt,auto,node distance=1cm,
  thick,main node/.style={circle, draw,font=\scriptsize\sffamily}]

  \node[main node] at (2,2) (1) {1};
  \node[main node] at (0,3) (2) {2};
  \node[main node] at (2,4) (3) {3};
  \node[main node] at (4,3) (4) {4};
  \node[main node] at (3,1) (5) {5};
  \node[main node] at (1,1) (6) {6};

 \draw (1) -- (2) ;
 \draw (3) -- (2) ;
 \draw (3) -- (4) ;
 \draw (4) -- (5) ;
 \draw (5) -- (6) ;
 \draw (1) -- (3) ;
 \draw (6) -- (2) ;
 \draw (1) -- (4) ;
 \draw (1) -- (5) ;
\end{tikzpicture}
 \end{center}
\caption{Illustrating when $\BZ$ does not preserve containment.}\label{fig5}
\end{figure}
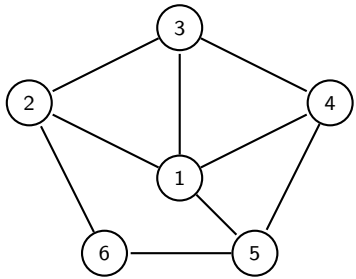

\begin{proof}
Let $G'$ be the graph in~\autoref{fig3}. Since $P = FRAC(G)$ and $P' = FRAC(G')$ and that $G'$ is a proper subgraph of $G$, it is easy to see that $P \subset P'$. We also showed in the proof of~\autoref{redundantineq} that $\BZ$ applied to the system $P'$ yields $P'_I$.

Next, if we apply $\BZ$ to $P$, then every tier has size $1$, and $\BZ(P) = \LS(P)$. Observe that the inequality $\sum_{i=1}^6 x_i \leq 2$ is valid for $P'_I = STAB(G')$. On the other hand, $y := \frac{1}{3}(1,1,1,1,1,2)^{\top}$ is in $\LS(P)$, certified by the following matrix in the lifted space:
\[
Y := \frac{1}{3}
\begin{pmatrix}
3 & 1 &1 &1&1&1&2 \\
1& 1&0 &0 &0 & 0& 1\\
1& 0& 1&0 & 0&1 &0 \\
1& 0& 0& 1&0 & 0& 1\\
1& 0&0 & 0&1 &0  &1 \\
1& 0& 1& 0& 0&1 & 0\\
2&1 & 0& 1 & 1 &0 &2
\end{pmatrix}.
\]
Since $\sum_{i=1}^6 y_i = \frac{7}{3} > 2$, we see that $\BZ(P) \not\subseteq \BZ(P')$.
\end{proof}

Finally, we provide the proof to \autoref{BZrankofFRACG}.

\begin{proof}[Proof of \autoref{BZrankofFRACG}]
Let $P := FRAC(K_n)$. We first prove the lower bound, by showing that all tiers generated by $\BZ'^k$ of size greater than $k+1$ are $P$-useless. This, combined with \autoref{SABZ'}, implies that $\BZ'^k(P) \supseteq \SA'^{2k+2}(\O_k(P))$.

Since the set of $k$-small obstructions of $FRAC(K_n)$ is exactly $E$ for every $k \geq 1$, we see that $\W_k = \set{W \subseteq [n] : |W| \leq k+1}$ and  $\T_k = \set{S \subseteq [n] : |S| \leq k(k+1)}$. Now if $S$ is any tier of size at least $k+2$, we see that $(S \sm T)|_1 \cap T|_0 \cap P = \es$ for all $T \subseteq S$ such that $|T| \leq k$. This is because in such cases $|S \sm T| \geq 2$, and there are no points in $P$ which contain at least two ones. Thus, the only variables $\a$ associated with $S$ such that $\a \cap P \neq \es$ take the form $(S \sm (T \cup U))|_1 \cap T|_0 \cap U|_{< |U| -(k-|T|)}$. However, in this case we know that $S \sm (T \cup U)$ has size zero or one, and thus $\a \cap P$ is equal to either $\F \cap P$ or $i|_1 \cap P$ for some $i \in [n]$. Therefore, all variables associated with $S$ are $P$-useless, and so the tier $S$ is $P$-useless.

Also, observe that $P = \O_k(P)$ for any $k \geq 1$, and $P$ is known to have $\SA$-rank $n-2$. In fact, the matrix that certifies $\frac{1}{n-1} \bar{e} \in \SA^{n-3}(P)$ also belongs to $\tilde{\SA}'^{n-3}(P)$. Hence, the $\SA'$-rank of $P$ is $n-2$ as well. Thus, it follows that the $\BZ'$-rank of $P$ is at least $\left\lceil \frac{n}{2}\right\rceil -2$. 
Moreover, since $\BZ'$ dominates $\BZ''$, it follows from \autoref{BZ''upper} that $FRAC(G)$ has $\BZ'$-rank at most $\left\lceil \frac{n+1}{2}\right\rceil$.

Finally, we turn to the $\BZ$-rank of $FRAC(G)$. Again, $\O_k = E$ for all $k \geq 1$. Therefore, in this case the conditions ($\BZ3$) and ($\BZ'3$) are equivalent. Since each vertex is incident with at least two edges, $\BZ$ does generate all the singleton sets as walls. Thus, the $\BZ$- and $\BZ'$-rank of $FRAC(G)$ must coincide.
\end{proof}

\end{document}